\documentclass[11pt, reqno]{amsart}
\usepackage{graphicx, amssymb, amsmath, amsthm}
\usepackage[utf8]{inputenc}
\usepackage{epsfig}
\usepackage{hyperref}
\usepackage{comment}
\usepackage{mathrsfs}
\numberwithin{equation}{section}
\usepackage{amsthm}
\usepackage{subfigure}
\usepackage{tikz}
\usepackage{here}
\usepackage{float}
\usepackage{pifont}
\usepackage{enumitem}

\usetikzlibrary{matrix,arrows}
\usetikzlibrary{shapes}
\usetikzlibrary{calc}
\usetikzlibrary{arrows}
\usetikzlibrary{decorations.pathreplacing,decorations.markings}

\usepackage{tikz-cd}

\usetikzlibrary{patterns}

\newtheorem{theorem}{Theorem}[section]

\newtheorem{lemma}[theorem]{Lemma}

\newtheorem{corollary}[theorem]{Corollary}

\newtheorem{proposition}[theorem]{Proposition}
\newtheorem{question}[theorem]{Question}

\theoremstyle{definition}
\newtheorem{definition}[theorem]{Definition}
\newtheorem{remark}[theorem]{Remark}

\makeatletter
\newcommand{\Extend}[5]{\ext@arrow0099{\arrowfill@#1#2#3}{#4}{#5}}
\makeatother

\DeclareMathOperator{\Lip}{Lip}
\DeclareMathOperator{\dist}{dist}
\DeclareMathOperator{\diam}{diam}
\DeclareMathOperator{\vol}{vol}
\DeclareMathOperator{\spt}{spt}
\DeclareMathOperator{\inj}{inj}
\DeclareMathOperator{\Ric}{Ric}
\DeclareMathOperator{\Rm}{Rm}
\DeclareMathOperator{\co}{co}
\DeclareMathOperator{\ddiv}{div}

\newcommand{\ve}{\varepsilon}

\begin{document}

\title[Curvature-free effects]{Curvature-free effects from volume growth and ends-counting and their applications}

\author[Y. Bi]{Yuchen Bi}
\address[Yuchen Bi]{Mathematical Institute, Department of Pure Mathematics, University of Freiburg, Ernst-Zermelo-Stra{\ss}e 1, D-79104 Freiburg im Breisgau, Germany}
\email{yuchen.bi@math.uni-freiburg.de}

\author[J. Zhu]{Jintian Zhu}
\address[Jintian Zhu]{Institute for Theoretical Sciences, Westlake University, 600 Dunyu Road, 310030, Hangzhou, Zhejiang, People's Republic of China}
\email{zhujintian@westlake.edu.cn}

\begin{abstract}
In this paper, we investigate two curvature-free effects from volume growth and ends-counting. Motivated by the problem of extending classical results under Ricci curvature assumptions to other curvature settings, we establish two main theorems. First, any complete non-compact manifold with sublinear lower volume growth admits a smooth bounded mean-concave exhaustion. Second, any complete manifold with infinitely many ends contains escaping geodesic lines outside every compact subset. As applications, we provide new proofs of the Calabi--Yau minimal volume growth theorem and the Cai--Li--Tam finite-ends theorem for nonnegative Ricci curvature, without relying on the Bishop--Gromov volume comparison theorem or analytic tools specific to Ricci curvature. We further extend these results to Riemannian manifolds with nonnegative scalar curvature and K\"ahler manifolds with positive holomorphic sectional curvature. As an important tool, we also provide a rigorous proof of Gromov's mean-convex smoothing theorem.
\end{abstract}

\maketitle

\section{Introduction}
In Riemannian geometry, curvature conditions provide powerful control on the global geometry of complete Riemannian manifolds. Among them, the Ricci curvature has been extensively studied via fundamental tools such as variation formulas, the Laplacian comparison theorem, the Bochner formula, and so on. Many deep results have been established by various authors. Examples include: length results such as the Bonnet--Myers--Cheng diameter theorem \cite{Myers41, Cheng75} and the Cheeger--Gromoll splitting theorem \cite{CG71}; volume results such as the Bishop--Gromov volume
comparison theorem and the Calabi--Yau minimal volume growth theorem \cite{Yau76}; and topological finiteness results such as the Cai--Li--Tam finite-ends theorem \cite{Cai91, Li-Tam92} and Liu's resolution of the three-dimensional Milnor conjecture \cite{Liu13}.

Although assumptions on scalar curvature or holomorphic sectional curvature do not, in general, furnish the same comparison tools as Ricci curvature bounds, they can still impose global restrictions in their respective settings. It is therefore natural to ask whether some of the Ricci-curvature effects
near infinity admit analogues under these weaker curvature assumptions.

This paper is motivated by the following fundamental question:
\begin{question}
Can volume growth and the number of ends be controlled simply by the scalar curvature or the holomorphic sectional curvature, without any direct assumption on the Ricci curvature?
\end{question}

To answer this question, we develop a {\it curvature-independent framework} to extract geometric properties from volume growth and the number of ends, which eventually unifies and extends classical results from the Ricci curvature to the scalar curvature and the holomorphic sectional curvature.

Actually, we are able to establish two curvature-free geometric results on the existence of mean-concave exhaustion and escaping geodesic lines from volume growth and the number of ends, respectively. These results hold for arbitrary complete non-compact Riemannian manifolds, independent of any curvature lower bounds.

\subsection{Mean-concave exhaustion}
Let $(M,g)$ be a complete non-compact Riemannian manifold. We say that $(M,g)$ has {\it sublinear lower volume growth} if
$$\liminf_{r\to +\infty}\frac{\vol_g(B_r(p))}{r}=0$$
for some fixed point $p\in M$.
Note that the choice of the point $p$ is irrelevant in this definition, as the lower limits of the volume ratio at different points coincide.

Throughout the paper, we will always take the following convention: for any smooth hypersurface $\Sigma$ in $(M,g)$ with unit normal $\nu$, the mean curvature $H$ of $\Sigma$ with respect to $\nu$ is defined by $$H=\ddiv_\Sigma \nu.$$
 We say that $\Sigma$ is mean-convex or mean concave with respect to $\nu$ if $H\geq 0$ or $H\leq 0$ respectively.

Our first main theorem establishes the existence of a smooth and bounded mean-concave open exhaustion of $M$ under the assumption of  sublinear lower volume growth.

\begin{theorem}\label{Thm: mean concave exhaustion}
    Assume that $(M,g)$ is a complete Riemannian manifold with sublinear lower volume growth. Then there exists an increasing exhaustion $\{U_l\}_{l\geq 1}$ of $M$ by smooth bounded open subsets such that each boundary $\partial U_l$ is mean-concave with respect to the unit normal pointing toward infinity.
\end{theorem}
\begin{remark}
The existence of a mean-concave exhaustion from sublinear lower volume growth was previously known up to dimension seven by \cite{Lott25} and an early preprint \cite{Zhu24} of this work, independently. In the present paper, we remove the dimension restriction by establishing the mean-convex smoothing theorem formulated by Gromov \cite{Gro14}, which is of independent interest. As a further application, we use this method to give another solution to Lawson's mean-convex approximation problem \cite[Problem~5.7]{Bro86}, previously solved by Wang \cite[Theorem~1.3]{Wan24}.
\end{remark}

\subsection{Escaping geodesic lines}
Let $M$ be a non-compact smooth manifold. Fix an increasing exhaustion $$\{K_l\}_{l\geq 1}$$ by smooth compact subsets. By an {\it end} of $M$, we mean a decreasing family $$\{V_l\}_{l\geq 1},$$ where each $V_l$ is an unbounded component of the complement $M\setminus K_l$. Denote the number of ends of $M$ by $$e(M).$$
We remark that $e(M)$  is a topological invariant of $M$, independent of the choice of the compact exhaustion $\{K_l\}_{l\geq 1}$. 

Our second main theorem establishes the existence of a geodesic line near infinity in the presence of infinitely many ends.
\begin{theorem}\label{Thm: escaping geodesic line}
    Let $(M,g)$ be a complete, non-compact Riemannian manifold with $$e(M)=+\infty.$$ Then for any compact subset $K\subset M$ there exists a geodesic line $$\gamma:(-\infty,+\infty)\to M\setminus K,$$
    where geodesic line means $$\dist_g(\gamma(s),\gamma(t))=|s-t|$$ for all $s,t\in \mathbb R$.
\end{theorem}

\subsection{Applications---old and new}
These two curvature-free results have significant applications. First, they yield alternative proofs of two classical results for nonnegative Ricci curvature:

\begin{theorem}[Calabi, Yau]
Let $(M,g)$ be a complete and non-compact Riemannian manifold with nonnegative Ricci curvature, then $(M,g)$ has at least linear volume growth. That is, we have
$$\liminf_{r\to +\infty}\frac{\vol_g(B_r(p))}{r}>0$$
for all $p\in M$.
\end{theorem}

\begin{theorem}[Cai, Li--Tam]
Let $(M,g)$ be a complete Riemannian manifold with nonnegative Ricci curvature outside a compact subset. Then $(M,g)$ has finitely many ends.
\end{theorem}

Crucially, our alternative proofs avoid the use of the Bishop--Gromov volume comparison theorem and other analytic tools specific to Ricci curvature, making the argument more robust and generalizable. As a consequence, we can generalize the minimal volume growth theorem in the setting of the scalar curvature on Riemannian manifolds and the finite-ends theorem in the setting of the holomorphic sectional curvature on K\"ahler manifolds.

\begin{theorem}\label{Thm: minimal volume scalar}
    Let $(M^n,g)$, $2\leq n\leq 7$, be a complete non-compact Riemannian $n$-manifold with nonnegative scalar curvature and compact boundary. If there is a compact set $K$ such that any closed smooth hypersurface separating $K$ from the infinity admits no smooth metric with positive scalar curvature, then $(M,g)$ has at least linear volume growth.
\end{theorem}
\begin{remark}
The dimension restriction comes from the construction of a separating stable $\mu$-bubble admitting positive scalar curvature to derive contradiction, and can be relaxed by using generic regularity theory from geometric measure theory. On the other hand, similar results can be established in all dimensions for spin manifolds having certain spinorial obstruction around infinity. For example, together with Theorem \ref{Thm: mean concave exhaustion}, the spinorial obstruction used in \cite[Section 2.3]{Lott25} yields a linear lower bound for the volume growth of \(M=\mathbb T^{n-1}\times[0,+\infty)\) with nonnegative scalar curvature in all dimensions. More generally, the same conclusion holds for \(N^{n-1}\times[0,+\infty)\), where \(N\) is any closed oriented enlargeable manifold, by adapting the proof of \cite[Theorem~D]{BiZhu26}. Indeed, the separating stable $\mu$-bubble is an enlargeable AM--PI space with a codimension-seven singular set. Although stability provides only a spectral inequality, the conformal blow-up and weight modification in \cite[Proposition~3.11 and the proof of Proposition~4.7, Step~4]{BiZhu26} convert it into the required pointwise weighted scalar-curvature inequality on the complete regular part. The remainder of that proof then applies.
\end{remark}

In particular, we have the following corollaries.
\begin{corollary}\label{Cor: contractible}
 Let $(M^3,g)$ be a complete and contractible Riemannian $3$-manifold with positive scalar curvature, which has sublinear lower volume growth. Then $M$ is homeomorphic to $\mathbb R^3$.
\end{corollary}
\begin{remark}
It is an open question whether a complete and contractible Riemannian $3$-manifold with positive scalar curvature is homeomorphic to $\mathbb R^3$. The interested audience can refer to \cite{Wang24,Wang25,BGM25,CLX26,Lott26} for backgrounds and related works.
\end{remark}

\begin{corollary}\label{Cor: aspherical at infinity}
Let $(M^n,g)$, $2\leq n\leq 5$, be a complete Riemannian manifold with nonnegative scalar curvature, which is aspherical outside some compact set $K$. Then $(M,g)$ has at least linear volume growth.
\end{corollary}

\begin{theorem}\label{Thm: finite ends holomorphic}
    Let $(M,g,J)$ be a complete and non-compact K\"ahler manifold with positive holomorphic sectional curvature outside a compact subset \(K\). Then we have 
    $$e(M)<+\infty.$$
\end{theorem}

The rest of this paper will be organized as follows.

In Section 2, we prove our first main theorem on the existence of smooth bounded mean-concave exhaustions. We develop the soap bubble method with an inner obstacle to construct minimizing $h$-hypersurfaces and derive mean-concave property for the minimizers. To handle the singularity issues in higher dimensions, we apply the mean-convex smoothing technique, whose rigorous proof is postponed to Section 3.

Section 3 is devoted to proving Gromov's mean-convex smoothing theorem, which is of independent interest and crucial for removing singularities in our construction of mean-concave exhaustion. We recall the theory of sets with positive reach from Federer's foundational work, and use the Sard-type theorems for distance functions to show that quasi-regular boundaries can be approximated by sets with positive reach. We then use viscosity barriers in the sense of Ilmanen, together with a prescribed-mean-curvature minimization argument and the Allard regularity, to obtain smooth approximations of mean-convex boundaries with controlled mean curvature. At the end of the section, we apply the smoothing argument to stationary boundaries, yielding a solution to Lawson's mean-convex approximation problem.

In Section 4, we give a proof of our second main theorem on the existence of escaping geodesic lines outside any compact subset.

Section 5 presents several applications of our curvature-free results. In the first part, we present new proofs of two classical theorems for manifolds with nonnegative Ricci curvature mentioned above. In the second part, we extend these results to broader curvature settings. 

\bigskip

\subsection*{Acknowledgements}
The second-named author is partially supported by National Key R\&D Program of China 2023YFA1009900, NSFC Grant No. 12401072, Zhejiang Provincial Natural Science Foundation of China under Grant No. LQKWL26A0101 and the start-up fund from Westlake University.

\section{Mean-concave exhaustion}
In this section, we always assume that $(M^n,g)$ is a complete non-compact Riemannian $n$-manifold with sublinear lower volume growth, that is,
$$\liminf_{r\to+\infty}\frac{\vol_g(B_r(p))}{r}=0$$
for some point $p\in M$.
 Let $K\subset M$ be a fixed compact subset. By enlarging $K$ if necessary, we may assume that $K$  is connected with smooth boundary, and that every connected component of the complement $M\setminus K$ is unbounded.

The goal of this section is to prove Theorem \ref{Thm: mean concave exhaustion}, which is an immediate consequence of  the following proposition.

\begin{proposition}\label{Prop: mean concave neighborhood}
There exists a smooth bounded open neighborhood $U$ of $K$ such that $\partial U$ is mean-concave with respect to the unit normal pointing toward infinity.
\end{proposition}

 Let us fix a component $E$ of $M \setminus K$ and consider the manifold 
 $$M_E = K \cup E.$$ It suffices to prove Proposition \ref{Prop: mean concave neighborhood} above for the one-ended manifold $(M_E, g)$ with smooth boundary. If $\dim M \leq 7$, then the construction of $U$ proceeds directly by formulating some minimization problem for a soap bubble with an inner obstacle, where the problem is carefully set up so that the minimizer has no contact with the obstacle after modification. In higher dimensions, we need to apply the mean-convex smoothing theorem first formulated by Gromov and rigorously proved in next section.

\subsection{The soap bubble method with inner obstacle}
 Let $V$ denote the compact manifold with boundary, bounded by the boundary $\partial E$ and a fixed smooth hypersurface, denoted by $\Sigma_{sep}$, that separates $\partial E$ from the infinity of $E$. In order to distinguish different boundary components, we denote
 $$\partial_-V:=\partial E\mbox{ and }\partial_+V:=\Sigma_{sep}.$$
To prevent the soap bubble from having contact with the inner obstacle, we need to introduce two positive constants. Define the {\it width constant} by
\begin{equation*}
c_{width} := \inf \left\{ \mathbb{M}_g(T) \left|
\begin{array}{c}
T \in \mathcal{I}_{n-1}(V)\mbox{ with } \partial T = 0\mbox{ and}\\
T \mbox{ is homologically nontrivial in } V
\end{array} \right. \right\},
\end{equation*}
where $\mathcal{I}_{n-1}(V)$ denotes the space of all integral $(n-1)$-currents supported in $V$. Define the {\it height constant} by
\begin{equation*}
c_{height} := \inf \left\{ \mathbb{M}_g(T) \left|
\begin{array}{c}
T \in \mathcal{I}_{n-1}(V)\mbox{ has mean curvature vector }|\vec H|\leq 1\\\mbox{ in } \mathring V,
\,\spt T \mbox{ is connected and } \spt T\cap \partial_\pm V\neq\emptyset
\end{array} \right. \right\}.
\end{equation*}
Roughly speaking, the width and height constants measure, respectively, the width and the height of $V$ in terms of area.

\begin{lemma}\label{Lem: positivity}
   Both $c_{width}$ and $c_{height}$ are positive.
\end{lemma}
\begin{proof}
 Note that we can take a perturbation $\tilde g$ of $g$ such that
 \begin{itemize}
 \item $\tilde g$ and $g$ are equivalent with $g/2\leq \tilde g\leq 2g$;
 \item and the boundary $\partial V$ is mean-convex with respect to $\tilde g$.
 \end{itemize}
 In particular, the width constant $\tilde c_{width}$ with respect to $\tilde g$ satisfies
 $$\tilde c_{width}\leq 2^{n-1}\cdot c_{width}.$$
 By the geometric measure theory, the width constant $\tilde c_{width}$ is attained by some nontrivial integral $(n-1)$-current in $V$, and is therefore positive, which then yields the positivity of $c_{\mathrm{width}}$. The positivity of $c_{\mathrm{height}}$ comes from the monotonicity formula, and the fact that $\spt T$ must pass through a fixed hypersurface separating $\partial_- V$ and $\partial_+V$.
\end{proof}

We will formulate the minimization problem as follows. Recall that $(M,g)$ is assumed to have sublinear lower volume growth, and then we are able to find a sequence of positive constants $r_i\to+\infty$ such that
$$
\frac{\vol_g(B_{r_i}(p))}{r_i}\to 0\mbox{ as }i\to\infty.
$$
Fix a smooth proper function $$\rho:M_E\to[0,+\infty)$$ satisfying 
\begin{itemize}
\item $\rho^{-1}(0)=K$;
\item $\Lip\rho\leq 2$;
\item and $|\rho(\cdot)-\dist(\cdot,K)|\leq 1$.
\end{itemize}
A straightforward computation shows
$$
r_i^{-1}\vol_g\left(\rho^{-1}\left(\left[0,\frac{r_i}{2}\right]\right)\right)\to 0\mbox{ as }i\to\infty.
$$
Applying the coarea formula, we obtain
$$
r_i^{-1}\int_{\frac{r_i}{4}}^{\frac{r_i}{2}}\mathcal H_g^{n-1}(\{\rho=\tau\})\,\mathrm d\tau=r_i^{-1}\int_{\{r_i/4\leq \rho\leq r_i/2\}}|\mathrm d\rho|_g\,\mathrm d\mathcal H^n_g\to 0\mbox{ as }i\to\infty.
$$
Consequently, for each $i$ we can choose a regular value $t_i$ of $\rho$ satisfying $$t_i\in(r_i/4,r_i/2)$$ such that
$$
\mathcal H^{n-1}_g(\{\rho=t_i\})\to 0\mbox{ as } i\to\infty.
$$
In particular, we may take $t_i$ sufficiently large so that $V\subset\{\rho<t_i\}$ and
$$\mathcal H^{n-1}_g(\{\rho=t_i\})<\min\{c_{width},c_{height}\},$$
which can be guaranteed from the positivity of the right-hand side by Lemma \ref{Lem: positivity}.
We fix such a $t_i$ in the following discussion. 

Note that we can take a small constant $\delta\in (0,1)$ such that
\begin{equation}\label{Eq: prescribed comparison}
  c_0<\min\{c_{width},c_{height}\},
\end{equation}
where
$$c_0:=  \mathcal H^{n-1}_g(\{\rho=t_i\})+\delta\cdot \mathcal H^n_g(\{\rho\leq t_i\}).$$

For our purpose, we need to take a slightly larger regular value $t_i^*>t_i$ of $\rho$ and set $$W:=\rho^{-1}([0,t_i^*]).$$ 
Fix a smooth function $\eta:[0,t_i^*)\to \mathbb R$ satisfying the following properties:
\begin{itemize}
    \item $\eta(t)\leq 0$ for all $t$;
    \item $\eta\equiv 0$ on $[0,t_i]$;
    \item and $\eta(t)\to -\infty$ as $t\to t_i^*$.
\end{itemize}
Define 
$$h:=\eta\circ\rho,$$ 
which will serve as the prescribed mean curvature function below.

Consider the class of Caccioppoli sets
$$
\mathcal C_E=\left\{
\begin{array}{c}
	\mbox{open Caccioppoli sets } \Omega\subset M_E\mbox{ such that }\\ K\subset \Omega\mbox{ and }\Omega\Subset W\setminus \rho^{-1}(t_i^*)
\end{array}\right\}
$$
and the functional
$$\mathcal A^{h,\delta}_E(\Omega)=\mathcal H^{n-1}_g(\partial^*\Omega\cap \mathring W)-\int_{\Omega}(h-\delta)\,\mathrm d\mathcal H^n_g\mbox{ with }\Omega\in \mathcal C_E,$$
where $\partial^*\Omega$ denotes the reduced boundary of $\Omega$ and $\mathring W$ denotes the interior of $W$.

\begin{lemma}\label{Lem: minimizer}
    There exists an $\Omega_E\in \mathcal C_E$ such that
    $$\mathcal A^{h,\delta}_E(\Omega_E)=\min_{\Omega\in \mathcal C_E}\mathcal A^{h,\delta}_E(\Omega).$$
\end{lemma}
\begin{proof}
    By the geometric measure theory we can always find a minimizer $\Omega_E$ of $\mathcal A^{h,\delta}_E$ in $\mathcal C_E$ as desired, see \cite[Proposition 2.1]{Zhu21} for instance. 
\end{proof}
Let $\Omega_E^*$ denote the component of $\Omega_E$ containing $K$, and $\widehat\Omega_{E}$ be the union of $\Omega_E^*$ and all bounded components of its complement in $M_E$. Then $\widehat\Omega_E$ is a bounded open set and satisfies $\partial\widehat\Omega_E \subset \partial\Omega_E$.

The following lemma says that $\partial \widehat\Omega_E$ has no contact with $K$.

\begin{lemma}\label{Lem: untouch component}
$K \subset \widehat\Omega_E$.
\end{lemma}
\begin{proof}
   Otherwise, $K$ must intersect some component of $\partial\widehat\Omega_E$, hence of $\partial\Omega_E$, denoted by $\mathcal S$. Note that $\mathcal S$ must be a common component of $\partial\Omega_E^*$ and $\partial U$ for some unbounded component $U$ of $M_E\setminus \bar \Omega_E^*$. We know that $\mathcal S$ is homologically non-trivial in $M_E\setminus K$. Using \eqref{Eq: prescribed comparison} together with a direct comparison of the functional $\mathcal A^{h,\delta}_E$ evaluated at $\{\rho \leq t_i\}$ and at $\Omega_E$, we obtain
    $$\mathbb M_g(\mathcal S)\leq \mathbb M_g(\partial\Omega_E)\leq c_0<c_{width}.$$
   This implies $$\mathcal S \setminus V \neq \emptyset.$$ A similar comparison argument using the height constant shows $$\mathbb M_g(\mathcal S) < c_{\text{height}},$$ which in turn implies $$\mathcal S \cap K = \emptyset,$$ contradicting to our assumption.
\end{proof}

In particular, we have the following mean curvature estimate.
\begin{corollary}
The regular part of $\partial \widehat\Omega_E\setminus K$ satisfies
$$H\leq -\delta$$
with respect to the outward-pointing unit normal.
\end{corollary}

\subsection{Smooth mean-concave region}
To handle singularity issues in higher dimensions, we use the following mean-convex smoothing theorem.

\begin{proposition}\label{Prop: mean-convex-smoothing}
Suppose that $\Sigma$ is compact and 
quasi-regular relative to $\Omega$ and that the mean curvature of $\mathbf\Lambda(\Omega, \Sigma)$ with respect to the outward-pointing unit normal satisfies $H \geq \delta$ for some constant $\delta > 0$.
Then there exists a sequence of open subsets $\Omega_l\subset \Omega$ such that
\begin{itemize}
\item we have 
$$\partial\Omega_l\cap \partial\Omega=\partial\Omega\setminus \Sigma,$$
and $\Sigma_l:= \partial \Omega_l\setminus \partial\Omega$ is a smooth hypersurface whose mean curvature with respect to the outward-pointing unit normal satisfies $$H\geq \delta-l^{-1};$$
\item $\Omega\setminus\Omega_l$ lies in $l^{-1}$-neighborhood of $\Sigma$, and in particular, $\Sigma_l \to \Sigma$ in the sense of Hausdorff distance as $l \to \infty$;
\item we also have $\mathcal H^{n-1}(\Sigma_l)\to \mathcal H^{n-1}(\mathbf \Lambda(\Omega,\Sigma))$ as $l\to \infty$.
\end{itemize}
\end{proposition}

Due to its independent interest, the rigorous proof of Proposition
\ref{Prop: mean-convex-smoothing} will be presented in next section, where we also include the definition of $\mathbf \Lambda(\Omega,\Sigma)$ and of the quasi-regularity therein. In the remaining part of this section, we will complete the proof of Proposition \ref{Prop: mean concave neighborhood} by assuming Proposition
\ref{Prop: mean-convex-smoothing}.

\begin{proof}[Proof of Proposition \ref{Prop: mean concave neighborhood}]
We first deal with the one-ended manifold $(M_E,g)$ and show that there is a smooth bounded open set $U$ containing $K$ such that $\partial U\setminus K$ is mean-concave.

Let $\Omega$ denote the exterior region of $\widehat \Omega_E$ and 
$$\Sigma=\partial \Omega.$$
Clearly, $\Sigma$ is compact.
We will show that $\mathbf \Lambda(\Omega,\Sigma)$ is contained in the regular part of $\partial \widehat\Omega_E$, so $\Sigma$ is quasi-regular relative to $\Omega$, and the hypersurface $\mathbf \Lambda(\Omega,\Sigma)$ satisfies $H\geq \delta$ with respect to the outward-pointing unit normal. To see this, we take any point $x\in \mathbf \Lambda(\Omega,\Sigma)$. By definition, we have $x\in\mathbf \Lambda_\rho(\Omega,\Sigma)$, which means that there exists a unit-speed minimizing geodesic 
$$\gamma: [0, \rho] \to M$$ such that 
\[
\gamma(\rho)=x,\, d(\gamma(0),\Sigma)=\rho\mbox{ and }\gamma([0,\rho))\subset\Omega.
\]
Clearly, the geodesic sphere $\partial B_{\rho/2}(\gamma(\rho/2))$ is smooth and touches $\mathbf \Lambda(\Omega,\Sigma)$ at $x$, which implies that $x$ is a regular point by a blow-up argument and the half-space property of area-minimizing hypersurfaces.

By Proposition \ref{Prop: mean-convex-smoothing}, we can find an open subregion $\Omega_l\subset \Omega$ such that
\begin{itemize}
\item $\partial\Omega_l$ satisfies $H\geq \delta/2>0$ with respect to the outward-pointing unit normal;
\item $\Omega\setminus \Omega_l$ lies in the $1$-neighborhood of $\partial\Omega$.
\end{itemize}
Take 
$$U=M\setminus \overline \Omega_l.$$
 Then $U$ is a smooth bounded region containing $K$ such that the boundary portion $\partial U\setminus K$ is mean-concave with respect to the outward-pointing unit normal.

For the general case, recall that $M\setminus K$ consists of unbounded components $E_1,E_2,\ldots,E_k$. From the previous discussion, we can find smooth bounded open subsets $U_j\subset K\cup E_j$ such that $\partial U_j\setminus K$ is mean-concave. Take
$$U=\bigcup_{j=1}^kU_j\subset M.$$
Then $U$ is a smooth bounded region containing $K$ with mean-concave boundary.
\end{proof}

Now Theorem \ref{Thm: mean concave exhaustion} follows immediately.
\begin{proof}[Proof of Theorem \ref{Thm: mean concave exhaustion}]
Take a smooth compact exhaustion $\{K_l\}_{l\geq 1}$ such that $M\setminus K_l$ consists of unbounded components. By Proposition \ref{Prop: mean concave neighborhood}, we can find smooth bounded regions $U_l$ containing $K_l$ such that $\partial U_l$ is mean-concave with respect to the outward-pointing unit normal. Clearly, $\{U_l\}_{l\geq 1}$ is still an exhaustion of $M$. Up to a subsequence, we can guarantee that $\{U_l\}_{l\geq 1}$ is increasing.
\end{proof}

\section{Mean-convex smoothing}\label{Sec: smoothing}
In this section, we will establish Proposition \ref{Prop: mean-convex-smoothing}. As an application, we prove a local mean-convex approximation result for stationary boundaries, thereby giving a solution to a problem of Lawson.

Let $\Omega \subset M$ be an open subset and $\Sigma$ be a union of some of the connected components of $\partial \Omega$. For each constant $\rho > 0$, we define 
$$\mathbf\Lambda_\rho(\Omega, \Sigma)$$
 to be the set of points $x \in \Sigma$, for which there exists a unit-speed minimizing geodesic $\gamma: [0, \rho] \to M$ such that 
\[
\gamma(\rho)=x,\, d(\gamma(0),\Sigma)=\rho\mbox{ and }\gamma([0,\rho))\subset\Omega.
\]
Denote 
\begin{equation}\label{Eq: projection set}
\mathbf\Lambda(\Omega,\Sigma):=\bigcup_{\rho>0}\mathbf\Lambda_\rho(\Omega,\Sigma).
\end{equation}

First let us recall the following notion of  %$C^k$-quasi-regular 
 quasi-regular boundary.
\begin{definition}
   The boundary portion $\Sigma$ is called  %$C^k$-quasi-regular
   {\it quasi-regular}
   relative to $\Omega$ if $\mathbf\Lambda(\Omega,\Sigma)$ is a
   smooth hypersurface.  
\end{definition}

%\begin{proposition}\label{Prop: mean-convex-smoothing}
%Suppose that $\Sigma$ is compact and %$C^k$-quasi-regular 
%quasi-regular relative to $\Omega$ and that the mean curvature of $\mathbf\Lambda(\Omega, \Sigma)$ with respect to the outward-pointing unit normal satisfies $H \geq \delta$ for some constant $\delta > 0$. %Then for every $\delta\in (0, \delta_0)$, there exists an open subset $\Omega_\delta\subset \Omega$ such that
%\begin{itemize}
%\item $\Sigma_\delta:= \partial \Omega_\delta\cap \Omega$ is a smooth hypersurface whose mean curvature with respect to the outward-pointing unit normal satisfies $H\geq \delta$;
%\item $\Sigma_\delta \to \Sigma$ in the sense of Hausdorff distance as $\delta \to \delta_0$.
%\end{itemize}
%Then there exist a sequence of open subsets $\Omega_i\subset \Omega$ such that
%\begin{itemize}
%\item $\Sigma_i:= \partial \Omega_i\cap \Omega$ is a smooth hypersurface whose mean curvature with respect to the outward-pointing unit normal satisfies $H\geq \delta-i^{-1}$;
%\item $\Sigma_i \to \Sigma$ in the sense of Hausdorff distance as $i \to \infty$.
%\end{itemize}

%\end{proposition}

Our proof is based on the theory of the sets with positive reach as well as the theory of the soap bubbles. The idea for mean-convex smoothing is illustrated in the Figure \ref{Fig: smoothing} below.

\begin{figure}[htbp]
\centering
\includegraphics[width=10cm]{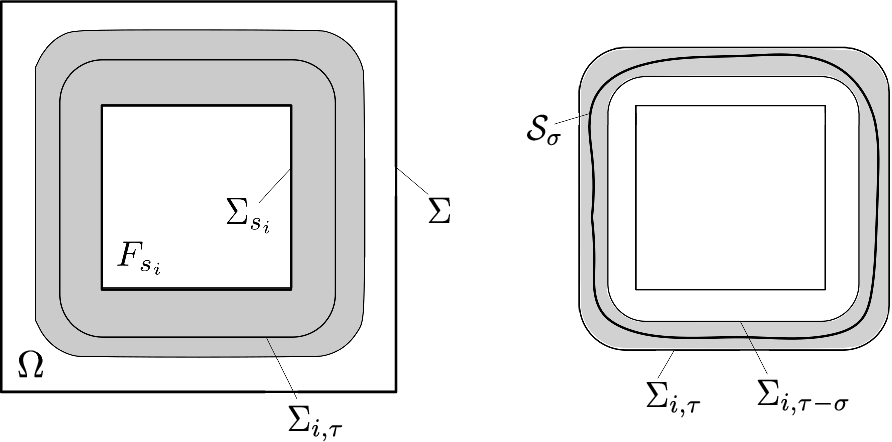}
\caption{There are two steps. In the first step, we deform $\Sigma$ inward to some equidistant hypersurface $\Sigma_{s_i}$ to it, and correspondingly the region $\Omega$ is deformed to $F_{s_i}$. The benefit is that $F_{s_i}$ is a set with positive reach for almost every $s_i$, and so the equidistant hypersurfaces $\Sigma_{i,\tau}$ are $C^{1,1}$ and still mean-convex. In the second step, we construct smooth mean-convex hypersurface by solving certain soap bubble $\mathcal S_\sigma$ in a thin layer between $\Sigma_{i,\tau}$ and $\Sigma_{i,\tau-\sigma}$. Here, the prescribed mean curvature function can be uniformly bounded, and then $\mathcal S_\sigma$ has almost density one everywhere. The Allard regularity and the elliptic estimates promise the desired smoothness.}
\label{Fig: smoothing}
\end{figure}

\subsection{Basics on sets with positive reach}
In the following, we recall some basic properties of sets with positive reach introduced by Federer \cite{Fed59}. 

Let $E$ always denote a closed subset of $(M,g)$ and define
$$d_E(p):=\dist(p,E).$$
Given a point $p \in M$, we use $\mathfrak C_p$ to denote the set consisting of nearest points in $E$ to $p$, where a point $q \in E$ is called a nearest point in $E$ to $p$ if $$d_E(p)=\dist(p,q).$$ 
Note that we have $$\mathfrak C_p\neq \emptyset$$ for all $p\in M$ since $E$ is closed. For convenience, let us denote
$$\mathcal U_E:=\{p\in M:\#\mathfrak C_p=1\}.$$
\begin{definition}
    The reach of $E\subset M$ at $q\in E$ is defined by
    $$\mathfrak r(E,q):=\sup\{r\geq 0:B_r(q)\subset\mathcal U_E\}.$$
    The reach of $E$ is defined by 
    $$\mathfrak r(E):=\inf_{q\in E}\mathfrak r(E,q).$$
    For convenience, $E$ is called a set with positive reach if $\mathfrak r(E)>0$.
\end{definition}
\begin{remark}\label{Rmk: boundary determine reach}
It is easy to check
$$\mathfrak r(E)=\inf_{q\in \partial E}\mathfrak r(E,q)$$
since for any $q\in E$ we have $$B_r(q)\setminus E\subset \bigcup_{q'\in \partial E\cap B_r(q)}B_r(q').$$
\end{remark}
From now on, we will always assume $E$ to be a closed set with positive reach.
The following lemma says that sets with positive reach admit a nice projection map in its neighborhood.
\begin{lemma}\label{Lem: continuous projection}
    Denote
    $$E_r=\{p\in M:d_E(p)\leq r\}.$$
    Then for any $r\in [0,\mathfrak r(E))$ the projection map
    $$\Pi_E:E_r\to E,\,p\mapsto q\in \mathfrak C_p$$
    is a continuous map.
\end{lemma}
\begin{proof}
    The map is well-defined since the set $\mathfrak{C}_p$ contains exactly one element for every $p \in E_r$. To establish the continuity of the projection map $\Pi_E$, it suffices to show that for any sequence $p_i \to p$ as $i \to \infty$, every subsequence of ${\Pi_E(p_i)}$ admits a further subsequence converging to $\Pi_E(p)$. Observe that the set $\{p_i\}_{i=1}^\infty$ is bounded, and hence the image set $\{\Pi_E(p_i)\}_{i=1}^\infty$ is bounded as well. Consequently, every subsequence of ${\Pi_E(p_i)}$ contains a convergent subsequence, whose limit is denoted by some $q \in E$. From the continuity of the distance function, we have $\dist(p, q) = d_E(p)$, which implies $q \in \mathfrak{C}_p$, and therefore $q = \Pi_E(p)$.
\end{proof}

The next several lemmas are devoted to show that sets with positive reach have nice tubular neighborhood. To proceed, we need the notion of tangent and normal cones of a set $E$ along its boundary.
\begin{definition}
    Let $q\in \partial E$. The {\it tangent cone} of $E$ at $q$ is defined by
    $$\mathcal T_qE:=\left\{u\in T_qM:u=\lim_{i\to\infty}u_i\mbox{ with }u_i\in t_i^{-1}\exp_q^{-1}(E)\mbox{ and }t_i\to 0\right\},$$
    and the {\it normal cone} of $E$ at $q$ is defined by
    $$\mathcal N_qE:=\{v\in T_qM:g(u, v)\leq 0\mbox{ for all }u\in \mathcal T_qE\}.$$
\end{definition}
\begin{remark}\label{Rmk: normal cone convex}
It is clear that $\mathcal N_qE$ is a convex cone.
\end{remark}
The following discussion will be based on the extra assumption
\begin{equation}
\mathfrak i(\partial E)>0,
\end{equation}
where
$$\mathfrak i(\partial E)=\inf_{q\in \partial E}\inj(M,g,q)$$
and $\mathfrak i(M,g,q)$ denotes the injective radius of $(M,g)$ at $q$.
 We denote
$$r_*:=\min\{\mathfrak r(E),\mathfrak i(\partial E)\}$$
and
$$\mathfrak A:=\{p\in M: d_E(p)\in(0,r_*)\}.$$

\begin{lemma}\label{Lem: unique pair}
    Given any point $p\in \mathfrak A$, there is a unique pair $(q,v)$ such that
    \begin{itemize}
        \item $q=\Pi_E(p)\in \partial E$,
        \item $v\in T_qM$ with $|v|=1$,
        \item and $p=\exp_q(d_E(p) v)$.
    \end{itemize}
    Moreover, we have $v\in \mathcal N_qE$.
\end{lemma}
\begin{proof}
    The unique pair $(q, v)$ is defined by taking the unique point $$q=\Pi_E(p) \in \mathfrak{C}_p$$ and letting $v$ be the initial velocity $\gamma'(0)$ of the unique unit-speed minimizing geodesic $$\gamma: [0, d_E(p)] \to (M, g)$$ connecting $q$ to $p$. Clearly, we have $$p=\gamma(d_E(p))=\exp_q(d_E(p)v).$$ 
    Note that we have $\gamma(t)\notin E$ for $t>0$, and so $q \in \partial E$.
    
    To see $v \in \mathcal{N}_q E$, we use the fact $v = -\nabla \operatorname{dist}(\cdot, p)(q)$ together with the Taylor expansion
   $$\dist(\exp_q(w),p)=d_E(p)-g(v,w)+O(|w|^2)\mbox{ as }w\to 0.$$
   Suppose by contradiction that $g(v, u) > 0$ for some $u \in \mathcal{T}_q E$. Then we can find a sequence $u_i \in t_i^{-1} \exp_q^{-1}(E)$ with $t_i \to 0$, $u_i\to u$, and $g(v, u_i) \geq c>0$. Therefore, we have
   $$\dist(\exp_q(t_iu_i),p)=d_E(p)-g(v,u_i)t_i+O(t_i^2)<d_E(p)$$
   for sufficiently large $i$, which leads to a contradiction.
\end{proof}
For any point $p \in \mathfrak{A}$, let $(q_p, v_p)$ denote the unique pair associated to $p$ from Lemma \ref{Lem: unique pair}. We introduce the vector field
\begin{equation}\label{Eq: continuous vector field}
V(p):=(\mathrm d\exp_{q_p})_{(d_E(p){v_p})}(v_p)\in T_pM.
\end{equation}
    By the continuity of the projection map from Lemma \ref{Lem: continuous projection} and the smoothness of the exponential map, $V$ is a continuous unit vector field on $\mathfrak{A}$.
    \begin{lemma}\label{Lem: exponential map}
        For any $p\in \mathfrak A$, the curve
        $$\gamma: (0,r_*)\to \mathfrak A,\,t\mapsto \exp_{q_p}(tv_p),$$
        is a maximal integral curve of the continuous vector field $V$ with the initial condition $$\gamma(d_E(p))=p.$$ Moreover, we have
        $$d_E(\gamma(t))=t\mbox{ for all }t\in (0,r_*).$$
    \end{lemma}
\begin{proof}
    From the Peano existence theorem, there exists a $C^1$-smooth integral curve $\gamma$ of $V$ with the initial condition $\gamma(d_E(p)) = p$. We claim $d_E(\gamma(t)) = t$ and $\gamma(t) = \exp_{q_p}(t v_p)$. To see this, we consider the function
    $$f(t)=d_E(\gamma(t)).$$
    Clearly, the function  $f$ is Lipschitz and satisfies $$f(d_E(p))=d_E(p).$$
     In particular, the function $f$ is differentiable almost everywhere. Let us work at a differential point $t_*$ of $f$. For short, we denote 
     $$d_*(\cdot) := \dist(\cdot, \Pi_E(\gamma(t_*))).$$
      Note that we have $d_* \geq d_E$ and the equality holds at $\gamma(t_*)$, so we have 
      $$f'(t_*)=(d_*\circ\gamma)'(t_*)=1.$$
     By integration, this implies $f(t) = t$ and so the integral curve $\gamma$ is defined on the maximal interval $(0, r_*)$. Let $q$ denote the limit point of $\gamma(t)$ as $t \to 0_+$. Note that we have $q\in E$ and also the relation 
    $$d_E(p) \leq \dist(p, q)\leq L(\gamma|_{(0,d_E(p))})=d_E(p).$$
    This implies that $q$ is the nearest point in $E$ to $p$ and that $\gamma$ is the unit-speed minimizing geodesic connecting $q$ to $p$. Therefore, we have
    $q = q_p\in \mathfrak C_p$ and $\gamma(t) = \exp_{q_p}(t v_p)$ by Lemma \ref{Lem: unique pair}.
\end{proof}

This lemma allows us to extend geodesics with nearest point fixed.
\begin{corollary}\label{Cor: nearest point preserving along geodesic}
    If $(q,v)$ satisfies $q\in \partial E$, $v\in T_qM$ with $|v|=1$, and $$\Pi_E(\exp_q(t_*v))=q$$ for some $t_*\in (0,r_*)$, then we have 
    $$\Pi_E(\exp_q(tv))=q\mbox{ and } d_E(\exp_q(tv))=t$$ for all $t\in (0,r_*)$.
\end{corollary}
\begin{proof}
    Let $p = \exp_q(t_* v)$. Then $p$ is a point in $\mathfrak A$. By Lemma \ref{Lem: unique pair}, it follows that $q = q_p$ and $v = v_p$. The desired conclusion then follows from Lemma \ref{Lem: exponential map}.
\end{proof}

The following lemma says that normal exponential map works as well for sets with positive reach.
\begin{lemma}\label{Lem: nearest point preserving along normal direction}
    We have $$\Pi_E(\exp_q(tv))=q\mbox{ and }d_E(\exp_q(tv))=t$$ for all $t\in (0,r_*)$, $q\in \partial E$ and $v\in \mathcal N_qE$ with $|v|=1$.
\end{lemma}
\begin{proof}
From $v\in \mathcal N_qE$ and $|v|=1$, we must have $$\gamma((0,\ve))\subset \mathfrak A$$ for some small constant $\ve>0$.  Take a sequence $$t_i\in (0,\ve)$$ with $t_i \to 0$ as $i\to\infty$. Denote $$p_i := \exp_q(t_i v)\in \mathfrak A$$ and $$q_i := \Pi_E(p_i)\in \partial E.$$ By Lemma \ref{Lem: unique pair}, we have $$p_i = \exp_{q_i}(d_E(p_i) v_i)$$ for some $v_i \in \mathcal N_{q_i}E$ with $|v_i| = 1$. Up to a subsequence, we may assume $$(q_i, v_i) \to (q_\infty, v_\infty)$$ as $i \to \infty$. 
Note that we have $\dist(q, q_i) \leq 2t_i$ and so $q_\infty = q$. 

We claim that $v_\infty = v$.
  To see this, fix a geodesic normal coordinate chart $\mathcal V$ centered at $q$ such that $v$ coincides with the coordinate vector $(\partial_n)_q$. Let $$\mathbb B_2 \subset \mathbb R^n$$ denote the Euclidean ball of radius two centered at the origin, and define
  $$\Psi_i : \mathbb B_2 \to \mathcal V,\,x\mapsto \exp_q(t_i x^j(\partial_j)_q).$$
  Set
  $$g_i:=\Psi_i^*(t_i^{-2}g).$$
 Then, up to a subsequence we have
 \begin{itemize}
 \item $g_i\to g_{euc}$;
 \item $\Psi_i^{-1}(q)=0$;
 \item $\Psi_i^{-1}(p_i)=e_n$;
 \item and $\Psi_i^{-1}(q_i)\to w$ as $i\to \infty$ for some $w\in \mathbb B_2$.
 \end{itemize}
  
 As the first step, let us show $w = 0$. Otherwise, we have $$e_n \cdot w > 0,$$ since we know $w\in B_1(e_n)$ from a distance comparison. In particular, for sufficiently large $i$, there exists a unique vector $u_i \in T_0 \mathbb B_2$ with $|u_i|_{g_i} = 1$ such that
  $$\Psi_i^{-1}(q_i)=\exp^{g_i}_0\left(s_iu_i\right),$$
  where $u_i\to w/|w|$ and $s_i=O(1)$ as $i\to\infty$. Decompose
  $$u_i=a_i^je_j\mbox{ and }\frac{w}{|w|}=a^je_j.$$
  It is clear that $a_i^j\to a^j$ as $i\to\infty$ and $a^n>0$. Note that we can write $$q_i=\exp_q(s_it_ia_i^j(\partial_j)_q),\mbox{ where }s_it_i\to 0\mbox{ as }i\to\infty.$$ 
  By definition, we have
  $$a^j(\partial_j)_q\in\mathcal T_qE$$ 
  from taking limit of $a_i^j(\partial_j)_q$. On the other hand, we have $$g(v,a^j(\partial_j)_q)=a^n>0,$$ which contradicts to the fact that $v \in \mathcal N_q E$. Hence $w = 0$.

  As the second step, we show $v_\infty=v$.  Decompose $$v_i = c_i^j (\partial_j)_{q_i}.$$ 
   Since the $g_i$-geodesics from $\Psi_i^{-1}(q_i)$ to $e_n$ converge to the Euclidean geodesic from $w$ to $e_n$, we must have $c_i^j e_j \to e_n$ as $i \to \infty$. This implies $c_i^j \to \delta^{jn}$ and so $v_\infty=v$. 
  
 We are ready to complete the proof for the whole lemma. Fix 
 $$t_* \in (0, r_*).$$ By the extension property from Corollary \ref{Cor: nearest point preserving along geodesic}, we have $$\Pi_E(\exp_{q_i}(t_* v_i)) = q_i.$$ Taking the limit as $i \to \infty$ yields $$\Pi_E(\exp_q(t_* v)) = q.$$ Applying Corollary \ref{Cor: nearest point preserving along geodesic} once again, we conclude that $\Pi_E(\exp_q(t v)) = q$ and $d_E(\exp_q(tv))=t$ for all $t \in (0, r_*)$.
\end{proof}

\subsection{From quasi-regularity to positive reach}

Let $\Omega\subset M$ be an open subset and $\Sigma$ be a union of some of the connected components of $\partial\Omega$, which is compact and quasi-regular relative to $\Omega$. For convenience, we use 
$$d_\Sigma:\Omega\to \mathbb R$$ to denote the distance function to $\Sigma$ in $(M,g)$ restricted to $\Omega$.

In the following, for each constant $s>0$ we denote
$$F_s:=\{q\in\Omega:\dist(q,\Sigma)\geq s\}$$
and
$$\Sigma_s:=\partial F_s\cap \Omega.$$
Denote
\begin{equation}\label{Eq: boundary distance gap}
s_*:=\dist(\Sigma,\partial\Omega\setminus\Sigma)>0.
\end{equation}
By taking $s<s_*$, we can guarantee that $\Sigma_s$ is compact in $\Omega$.

%Note that we can always guarantee
%\begin{equation}\label{Eq: near Sigma}
%\dist(\partial F_s,\partial\Omega\setminus\Sigma)>\frac{1}{2}\dist(\Sigma,\partial\Omega\setminus \Sigma).
%\end{equation}
%by taking $s$ small enough. In particular, $\Sigma_s$ is compact in $\Omega$.

We are going to show that $F_s$ has positive reach in $\Omega$ for almost every $s$ based on the Sard-type theorem for distance function by Rifford, combined with the characterization result for sets with positive reach by Bangert.

As a preparation, let us recall several different notions of differentials and related results. First we recall the definition of the Clarke differential, which is used in the Sard-type theorem for distance functions.
\begin{definition}
    Let $f$ be a locally Lipschitz function on $\Omega$, and $\mathcal D_f$ be the differentiable set of $f$. The Clarke differential \cite{Cla75} of $f$ at $q$ is defined by
    $$(\partial_Cf)(q):=\co\left\{\lim_{q_i\to q}(\nabla_g f)(q_i):q_i\in \mathcal D_f\right\},$$
where we use $\co(A)$ to denote the convex hull generated by the set $A$.
\end{definition}
\begin{remark}
    It is clear that $(\partial_Cf)(q)$ is a bounded closed convex set in $T_q\Omega$.
\end{remark}

%We note that the proof in \cite[Section 3.2]{Rif04} of the Sard-type theorem for the distance function to a smooth submanifold \cite[Theorem 1]{Rif04} actually applies without change to the distance function to a quasi-regular boundary.
We note that the Sard-type theorem for the distance function to a smooth submanifold \cite[Theorem 1]{Rif04} even holds the distance function to a quasi-regular boundary.

\begin{lemma}\label{Lem: Sard}
    If $\Sigma$ is quasi-regular relative to $\Omega$, then we have $$0\notin(\partial_Cd_\Sigma)(q)$$ at any point $q\in\Sigma_s$ for almost every $s\in (0,s_*)$.
\end{lemma}
\begin{proof}
We explain how to reduce to the smooth case, where the key observation is that the distance function $d_\Sigma$ near $\Sigma$ is actually realized by $\mathbf \Lambda(\Omega,\Sigma)$. To see this, we take any point $p\in \Omega$ with $d_\Sigma(p)<s_*$. We use $\Pi(p)$ to denote the collection of the nearest point in $\Sigma$ to $p$. For any point $q\in \Pi(p)$, there is a unit-speed minimizing geodesic $$\gamma:[0,s]\to (M,g)$$ connecting $p$ to $q$. Since we have $s<s_*$, it follows $$\gamma([0,s))\subset \Omega$$ and so $q\in \mathbf\Lambda(\Omega,\Sigma)$. To sum up, we already know
$$d_\Sigma(\cdot)=\dist(\cdot,\mathbf\Lambda(\Omega,\Sigma))$$
in $d_\Sigma^{-1}((0,s_*))$, and the lemma follows immediately from the Sard-type result \cite[Theorem 1]{Rif04}.
\end{proof}

Next we investigate differentials for semi-concave functions.
\begin{definition}
    A function $f$ on $\Omega$ is {\it concave} if the composed function 
    $f\circ \gamma$ is concave along every geodesic $\gamma$ in $\Omega$. A function $f$ on $\Omega$ is {\it semi-concave} if for any $q\in\Omega$ we can find a small geodesic ball $B_q$ centered at $q$ and a smooth function $h$ on $B_q$ such that $f+h$ is a concave function on $B_q$.
\end{definition}

For semi-concave functions we can introduce the super-differential at each point as below.

\begin{definition}
    Let $f$ be a semi-concave function on $\Omega$. Then the super-differential of $f$ at $q$ is defined by
    $$(\partial_Sf)(q):=\{w\in T_qM:\partial_vf\leq w\cdot v\mbox{ for all }v\in T_qM\},$$
    where $\partial_vf$ denotes the directional derivative of $f$ along $v$.
\end{definition}

The following lemma shows that the Clarke differential coincides with the super-differential for semi-concave functions.
\begin{lemma}\label{Lem: two gradients are same}
    Let $f$ be a semi-concave function on $\Omega$. Then $f$ is a locally Lipschitz function and we have $\partial_Cf=\partial_Sf$ at every point $q\in \Omega$.
\end{lemma}
\begin{proof}
    Since we are verifying local properties, we may assume $f$ itself to be concave for simplicity. In this case, the values of $f$ near any point $q$ can be bounded from below by convex combinations of finitely many values $f(q_i)$ at points $q_i$ near $q$, as in the Euclidean setting. Consequently, there exists a small geodesic ball $B_q$ centered at $q$ of some small radius $\varepsilon$ such that 
    $$f \ge c$$ in $B_q$ for some constant $c$. 
    
    On the other hand, since the function $f$ is concave along geodesics, for every unit vector $v \in T_qM$ and $s \in [0,\varepsilon)$, we have
    $$f(\exp_q(sv))\leq f(q)+ \frac{f(q)-f(\exp_q(-\ve v))}{\ve}s.$$
    In particular, $f$ is also bounded from above in $B_q$. Therefore, we obtain 
    $$|f|\leq C$$ in $B_q$ for some constant $C$.

    Let $B_q^*$ denote the geodesic ball of radius $\varepsilon/2$ centered at $q$. For any two points $q_1, q_2 \in B_q^*$, we take $\gamma$ to be the unit‑speed minimizing geodesic from $q_1$ to $q_2$. Extending the geodesic $\gamma$ beyond $q_2$ by length $\varepsilon/2$, the endpoint $q_3$ lies in $B_q$. By concavity of $f$ along geodesics,
    $$\frac{f(q_2)-f(q_1)}{\dist(q_1,q_2)}\geq \frac{f(q_3)-f(q_2)}{\ve/2}\geq -\frac{4C}{\ve}.$$
    Since $q_1$ and $q_2$ are arbitrary, we see that $f$ is locally Lipschitz, and so $\partial_Cf$ is well-defined.
    
    In the following, let us verify that $$(\partial_C f)(q) = (\partial_S f)(q)$$ at every point $q\in \Omega$.
    Note first that, by definition, $(\partial_S f)(q)$ is an intersection of closed half-spaces in $T_qM$, and hence is closed and convex.
    First we show $(\partial_Cf)(q) \subset (\partial_Sf)(q)$. It suffices to consider
\[
    w=\lim_{i\to\infty}(\nabla_g f)(q_i),
    \qquad q_i\in D_f,\quad q_i\to q.
\]
 Since $f$ is concave along geodesics, we have
    $$f(\exp_{q_i}(tv))\leq f(q_i)+t(\nabla_gf)(q_i)\cdot v\mbox{ for all }v\in T_{q_i}M.$$
    Passing to the limit yields
    $$f(\exp_q(tv))\leq f(q)+tw\cdot v\mbox{ for all }v\in T_qM.$$
    This implies $w \in (\partial_S f)(q)$. Taking convex hulls yields \((\partial_C f)(q)\subset(\partial_S f)(q)\).

Next we show $(\partial_Sf)(q)\subset (\partial_Cf)(q)$. Otherwise, we can take $w \in (\partial_S f)(q)$ but $w \notin (\partial_C f)(q)$. Recall that $(\partial_Cf)(q)$ is a bounded closed convex set in $T_q\Omega$.
By the separating hyperplane theorem, there exists a vector $v \in T_q \Omega$ such that 
$$w\cdot v<\min_{w_*\in(\partial_Cf)(q)}w_*\cdot v.$$
 From the continuity, there exist a closed neighborhood $V$ of $v$ and a positive constant $\delta$ such that
\begin{equation}\label{Eq: convex separating}
    w\cdot v_*<\min_{w_*\in(\partial_Cf)(q)} w_*\cdot v_*-\delta
\end{equation} 
for all $v_*\in V$. Let us fix a sequence of positive constants $t_i\to 0$. Note that we can pick $v_i\in V$ such that $$q_i:=\exp_q(t_iv_i)\in \mathcal D_f$$ since $\mathcal D_f$ is dense. Up to a subsequence, we may assume $$(\nabla_gf)(q_i)\to w_*\in (\partial_Cf)(q)$$ and 
$$v_i\to v_\infty\in V\mbox{ as }i\to\infty.$$
Since $f$ is concave along geodesics and it is locally Lipschitz, 
we have
\[
    f(q)\leq f(q_i)-t_i(\nabla_gf)(q_i)\cdot v_i= f(q_i)-t_i w_*\cdot v_\infty+o(t_i).
\]
    On the other hand, by the locally Lipschitz of $f$ and the definition of $\partial_Sf$, we have 
    $$f(q_i)=f(\exp_q(t_iv_\infty))+o(t_i)\leq f(q)+t_iw\cdot v_\infty+o(t_i).$$
    For $i$ large enough we obtain a contradiction to \eqref{Eq: convex separating}.
\end{proof}

\begin{corollary}\label{Cor: closure once more}
    We have $\partial_S^*f=\partial_Sf=\partial_Cf$ at every point $q\in \Omega$, where
   \[
    (\partial_S^*f)(q)
    :=
    \operatorname{co}
    \left\{v\in T_qM:v=\lim_{i\to\infty}v_i,\quad v_i\in(\partial_Sf)(q_i),\quad q_i\to q \right\}.
\]
\end{corollary}
\begin{proof}
    By Lemma \ref{Lem: two gradients are same}, we know that $\partial_S f$ consists of convex combination of the limits of $\nabla_g f$ at points in $\mathcal D_f$. Since $\partial_S^* f$ is obtained by taking the convex hull of a further limit from $\partial_S f$, a diagonal argument yields $\partial_S^* f = \partial_S f$.
\end{proof}

We have to use the following characterization result established by Bangert \cite{Ban82} concerning the derivative $\partial_S^*f$. For convenience, we say that $c$ is a {\it regular value} of $f$ if 
$$0\notin (\partial_S^*f)(q)$$ 
for all $q\in f^{-1}(c)$.

\begin{lemma}[\cite{Ban82}] \label{Lem: regular value to positive reach}
    Let $f$ be a semi-concave function on $\Omega$. Then the following statements are equivalent:
    \begin{itemize}
        \item the reach of $E=\{f\geq c\}$ in $\Omega$ satisfies $$\mathfrak r_{\Omega}(E,p)>0$$  for all $p\in E$;
        \item $c$ is a regular value of $f$.
    \end{itemize}
\end{lemma}

We are ready to establish the following modification lemma.
\begin{lemma}\label{Lem: positive reach modification}
   Let $\Omega\subset M$ be an open subset and $\Sigma$ be a union of some of the connected components of $\partial\Omega$. Assume that $\Sigma$ is compact and quasi-regular relative to $\Omega$. Then there is a sequence of positive constants $s_i\to 0$ such that the set $$F_{s_i}:=\{q\in\Omega:\dist(q,\Sigma)\geq s_i\}$$
    has positive reach in $\Omega$.
\end{lemma}

\begin{proof}
By Lemma \ref{Lem: Sard}, we can choose a sequence of constants $s_i \to 0$ such that $$0 \notin (\partial_C d_\Sigma)(q)$$ for every $q \in d_\Sigma^{-1}(s_i)$. In the following, we are going to verify that $d_\Sigma$ is semi-concave on $\Omega$. Consequently, every $s_i$ is a regular value of $d_\Sigma$ by Corollary \ref{Cor: closure once more}, and then Lemma \ref{Lem: regular value to positive reach}  implies $$\mathfrak r_{\Omega}(F_{s_i}, p) > 0$$ for all $p \in F_{s_i}$.
    
    Fix a point $q \in \Omega$ and a small constant $\varepsilon > 0$ such that $$\operatorname{dist}(q, \Sigma) > 3\varepsilon.$$ Let $B_q$ be the geodesic ball of radius $\varepsilon$ centered at $q$. By the classical Hessian comparison theorem, there is a universal constant $C$ such that we have
    $$\nabla^2 \dist(\cdot,q_*)\leq Cg\mbox{ on }\partial B_\varepsilon(q_*)$$
    for all $q_*$ in the geodesic ball of radius $2\varepsilon$ centered at $q$. With the constant $C$ determined,  we may take the function $h$ to be a sufficiently large (negative) multiple of the squared distance to $q$ such that
    $$\nabla_g^2 h \le -C g\mbox{ in }B_q.$$
    
    We claim that the function $$f := d_\Sigma + h$$ is concave on $B_q$. Otherwise, we can find a geodesic
    $$\gamma:[0,1]\to B_q$$
   and a small positive constant $\delta$ such that
    $$(f\circ\gamma)(t)-(1-t)(f\circ\gamma)(0)-t(f\circ\gamma)(1)+\delta t(1-t)$$
    attains a negative minimum at some point $t_*\in (0,1)$. We take a unit‑speed minimizing geodesic 
    $$\gamma_{*}:[0,l]\to (M,g)$$ 
    connecting $\Sigma$ to $\gamma(t_*)$, and define
    $$f_*(\cdot):=l-\varepsilon+\dist(\cdot,q_*)+h\mbox{ with }q_*=\gamma_*(l-\varepsilon).$$
    From the triangle inequality, we have $f_*\geq f$ with the equality attained at the point $\gamma(t_*)$. In particular, the function
    $$(f_*\circ\gamma)(t)-(1-t)(f_*\circ\gamma)(0)-t(f_*\circ\gamma)(1)+\delta t(1-t)$$
    attains its negative minimum at $t_*$ as well. This implies
    $$(\nabla^2 f_*)(\gamma'(t_*),\gamma'(t_*))\geq 2\delta.$$
    On the other hand, we have
    $$\nabla^2 f_*=\nabla^2\dist(\cdot,q_*)+\nabla^2h\leq 0\mbox{ at }\gamma(t_*),$$
    and this leads to a contradiction.

We have shown that $d_\Sigma$ is a semi-concave function on $\Omega$, and so we have $\mathfrak r_{\Omega}(F_{s_i},p)>0$ for all $p\in F_{s_i}$. By definition, the function
\[
    p\longmapsto r_\Omega(F_{s_i},p)
\]
is $1$-Lipschitz. Since $\Sigma_{s_i}$ is compact in $\Omega$ for $s_i$ small enough, It follows from Remark \ref{Rmk: boundary determine reach} that we have
$$\mathfrak r_\Omega(F_{s_i})=\inf_{p\in\partial F_{s_i}}\mathfrak r_\Omega(F_{s_i},p)>0.$$
   That is, $F_{s_i}$ is a set with positive reach in $\Omega$.
\end{proof}

%\begin{corollary}
  %  For each $i$ there is a constant $\tau_i>0$ such that $\Sigma_{i,\tau}$ is a $C^{1,1}$-hypersurface for all $\tau \in (0,\tau_i)$.
%\end{corollary}
%\begin{proof}
  %  Since $\partial F_{s_i}\cap \Omega$ is compact, we have
   % $$\mathfrak i_\Omega(\partial F_{s_i}):=\inf_{q\in \partial F_{s_i}}\inj(\Omega,g,q)>0.$$
   % The desired consequence comes from Lemma \ref{Lem: level set regularity}.
%\end{proof}

\subsection{Mean-convex smoothing}
To conduct the mean-convex smoothing, we have to use the barriers in the viscosity sense first introduced by Ilmanen \cite{Ilm}.

\begin{definition}
Let $U \subset M$ and $\Gamma \subset U$. We say that $\Gamma$ is a {\it $c$-barrier} in $U$ if for every open set $D \subset U \setminus \Gamma$ such that $\partial D$ is smooth around $\partial D\cap \Gamma$, the mean curvature of $\partial D$ satisfies $H_{\partial D} \ge c$ along $\partial D \cap \Gamma$ with respect to the outward-pointing unit normal, where we call $D$ a {\it test region} of $(\Gamma,U)$ for simplicity. Moreover, we say that $\Gamma$ is a {\it strict $c$-barrier} in $U$ if $\Gamma$ is a $c'$-barrier in $U$ with $c'>c$.
\end{definition}

As before, let $\Omega\subset M$ be an open subset and $\Sigma$ be a union of some of the connected components of $\partial\Omega$, which is compact and quasi-regular relative to $\Omega$. Assume further that the mean curvature of the hypersurface $\mathbf\Lambda(\Omega, \Sigma)$ with respect to the outward-pointing unit normal satisfies $$H \geq \delta$$ for some constant $\delta > 0$.

\begin{lemma}\label{Lem: barrier Sigma}
    $\Sigma$ is a $\delta$-barrier in $\Omega\cup \Sigma$.
\end{lemma}
\begin{proof}
    Take $D$ to be any test region of $(\Sigma,\Omega\cup\Sigma)$. Since $D$ is smooth around $\partial D\cap \Sigma$, we have $p\in \mathbf\Lambda(D,\partial D)$ for all $p\in \partial D\cap \Sigma$, and so $p\in \mathbf\Lambda(\Omega,\Sigma)$. From the maximum principle we see $$H_{\partial D}(p)\geq H_{\mathbf\Lambda(\Omega,\Sigma)}(p)\geq \delta.$$
    This means that $\Sigma$ is a $\delta$-barrier in $\Omega\cup \Sigma$.
\end{proof}

\begin{lemma}\label{Lem: barrier F_s}
    For any constant $0<\delta'<\delta$ the boundary $\partial F_{s_i}$ is a $\delta'$-barrier in $F_{s_i}$ for $i$ large enough.
\end{lemma}
\begin{proof}
    Suppose not, then for any $i$ we can find a test region $D$ of $(\partial F_{s_i},F_{s_i})$ such that $H_{\partial D}(p)<\delta'$ for some $p\in\partial D\cap \partial F_{s_i}$. By intersecting with a small geodesic ball centered at $p$, we can guarantee
    $$\dist(D,\partial\Omega\setminus \Sigma)>s_i.$$
    In particular, there is a unique point $p_*\in \mathbf \Lambda_{s_i}(\Omega,\Sigma)$ such that $\dist(p,p_*)=s_i$.
    Denote
    $$D_{s_i}=\{q\in\Omega:\dist(q,D)< s_i\}.$$
    Then we must have $p_*\in\partial D_{s_i}\cap \Sigma$. If $p_*$ is not a focal point of $\partial D$, then $\partial D_{s_i}$ is smooth around $p_*$ and can be modified to be a test region of $(\Sigma,\Omega\cup\Sigma)$. It follows from the Riccati equation that
    $$H_{\partial D_{s_i}}(p_*)\leq H_{\partial D}(p)-C_*s_i,$$
    where $C_*$ denotes the lower bound of the Ricci curvature in $s_*$-neighborhood of $\Sigma$. For $i$ large enough, $s_i$ is sufficiently small and so we have 
    $$H_{\partial D_{s_i}}(p_*)<\delta,$$ which contradicts to Lemma \ref{Lem: barrier Sigma}. If $p_*$ is a focal point of $\partial D$, then we can slightly modify $D$ to make the previous argument still work.
\end{proof}

Denote
$$\Omega_{i,\tau}=\{q\in \Omega:\dist(q,F_{s_i})< \tau\}$$
and 
$$\Sigma_{i,\tau}=\partial \Omega_{i,\tau}\cap \Omega.$$

The following lemma gives us the first regularity improvement from sets with positive reach.
\begin{lemma}\label{Lem: C11 hypersurface}
Denote
$$\tau_*=\min\{\mathfrak r_\Omega(F_{s_i}),\mathfrak i_\Omega(\partial F_{s_i})\}>0,$$
where $\mathfrak i_{\Omega}(\partial F_{s_i})$ denotes the injective radius infimum of the region $(\Omega,g)$ along $\partial F_{s_i}$.
Then $\Sigma_{i,\tau}$ is a $C^{1,1}$-hypersurface for all $\tau\in (0,\tau_*)$. Moreover, the second fundamental form of $\Sigma_{i,\tau}$ with respect to the outward-pointing unit normal is uniformly bounded from below if $\tau$ has a gap to $\tau_*$, and is uniformly bounded from above if $\tau$ has a gap to $0$.
\end{lemma}
\begin{proof}

For any point $p\in \Sigma_{i,\tau}$ with $0<\tau<\tau_*$, by Lemma \ref{Lem: unique pair} we can find a unique pair $(q_p,v_p)$ such that $p=\exp_{q_p}(\tau v_p)$. Denote $$\gamma(t):=\exp_{q_p}(t v_p)$$ for short. By Corollary \ref{Cor: nearest point preserving along geodesic} we have  
$$\dist(\gamma(t),F_{s_i})=t\mbox{ for all }t\in (0,\tau_*).$$
Fix a constant $\varepsilon$ such that $$0<\tau-\varepsilon<\tau+\varepsilon<\tau_*.$$ Then we know that the geodesic spheres $\partial B_\varepsilon(\gamma(\tau-\varepsilon))$ and $\partial B_\varepsilon(\gamma(\tau+\varepsilon))$ touch $\Sigma_{i,\tau}$ at $p$ from the interior and the exterior of $\Omega_{i,\tau}$, respectively. Since $p$ is arbitrary, $\Sigma_{i,\tau}$ is a $C^{1,1}$-hypersurface from the Hessian comparison theorem. The final statement is a direct consequence from our argument.
\end{proof}

We have the following mean curvature estimate for %\(\partial\Omega_{i,\tau}\).
$\Sigma_{i,\tau}$.

\begin{lemma}\label{Lem: barrier Omega i,tau}
    For any constant \(0<\delta'<\delta\), $\Sigma_{i,\tau}$ is a \(\delta'\)-barrier in
    \(\overline{\Omega}_{i,\tau}\) for \(i\) large enough and \(\tau\) small
    enough.
\end{lemma}

\begin{proof}
     By Lemma \ref{Lem: barrier F_s} we can take \(i\) large enough such that
     \(\Sigma_{s_i}\) is a \(\delta''\)-barrier in \(F_{s_i}\) for some fixed
     \[
         \delta'+C_*s_i<\delta''<\delta,
     \]
     where \(C_*\) denotes the Ricci curvature lower bound in the
     \(s_*\)-neighborhood of \(\Sigma\).

     In the following, let us show that %\(\partial \Omega_{i,\tau}\)
    $\Sigma_{i,\tau}$ is a
     \(\delta'\)-barrier in \(\overline\Omega_{i,\tau}\) for \(\tau\) small
     enough. Suppose not, then there is a test region \(D\) of
     \((\Sigma_{i,\tau},\overline\Omega_{i,\tau})\) such that
     \(H_{\partial D}(p)<\delta'\) for some
     \(p\in \partial D\cap \Sigma_{i,\tau}\). Since \(F_{s_i}\) has
     positive reach, \(p\) has a unique nearest point \(q\) in \(F_{s_i}\) to
     \(p\) when \(\tau\) is small enough. In particular, we have
     \[
         p=\exp_q(\tau v)
     \]
     for some \(v\in \mathcal N_qF_{s_i}\) with \(|v|=1\).

     We make the following discussions according to the structure of
     $$\mathcal N_q^1F_{s_i}:=\{w\in \mathcal N_qF_{s_i}:|w|=1\}.$$
     
  \noindent   {\bf Case 1:} \(\mathcal N_q^1F_{s_i}=\{v\}\). 
     
   \vspace{3mm}  In this case, we can take the nearest point \(q_\Sigma\) in \(\Sigma\) to
         \(q\) and a unit-speed minimizing geodesic
         \[
             \gamma:[0,s_i]\to \overline\Omega
         \]
         connecting \(q\) to \(q_\Sigma\). It is easy to check $\gamma'(0)\in \mathcal N_q^1F_{s_i}$ and so
         \(\gamma'(0)=v\). This yields that
         \(\partial D_{s_i-\tau}\) touches \(\Sigma\) at the point
         \(q_\Sigma\), where
         \[
             D_{s_i-\tau}:=\{x\in\Omega:\dist(x,D)<s_i-\tau\}\subset \Omega.
         \]
         As before, from the Riccati equation we have
         \[
             H_{\partial D_{s_i-\tau}}(q_\Sigma)
             <\delta'+C_*(s_i-\tau)<\delta,
         \]
         which leads to a contradiction to Lemma \ref{Lem: barrier Sigma}.

          \vspace{3mm} \noindent{\bf Case 2:} \(\mathcal N_q^1F_{s_i}=\{v,-v\}\).  \vspace{3mm}

          From the same argument as in Case 1, we know that one vector in $\mathcal N_q^1F_{s_i}$, say $v$, is
         realized by the initial speed of a minimizing geodesic from \(q\) to \(\Sigma\). 
         
         We claim
         that \(-v\) is also realized by a minimizing geodesic from \(q\) to
         \(\Sigma\). Note that the geodesic $\zeta(t)$ from $q$ along $-v$ satisfies $\zeta(t)\notin F_{s_i}$ for $t>0$ small enough. In particular, we have $d_\Sigma(\zeta(t))<s_i$.  Applying
         the same argument as in the proof of Berger's lemma
         \cite[Chapter 13, Lemma 4.1]{doCarmo92}, there is a sequence $t_j\to 0$ such that any minimizing geodesic $\gamma_j$ connecting $\zeta(t_j)$ to $\Sigma$ satisfies $$\langle\gamma_j'(0), \zeta'(t_j)\rangle\geq 0.$$
         %with the short segment
        % joining \(q_j\) to \(q\) in place of the curve \(\lambda\),
        Letting $j\to\infty$, we obtain
         a limit minimizing geodesic from \(q\) to \(\Sigma\) whose initial speed \(w\) satisfies
         \[
             \langle w,v\rangle\leq0.
         \]
         Since \(w\in \mathcal N_q^1F_{s_i}\), it
         follows \(w=-v\). This proves the claim.

         Since both \(v\) and \(-v\) are realized by minimizing geodesics from
         \(q\) to \(\Sigma\), the Clarke differential
         \((\partial_C d_\Sigma)(q)\) contains two opposite unit vectors, and so we have
         \(
             0\in(\partial_C d_\Sigma)(q)
         \),
         contradicting our choice of \(s_i\) as a regular value of
         \(d_\Sigma\).

      \vspace{3mm}
     \noindent {\bf Case 3:} \(\mathcal N_q^1F_{s_i}\) has dimension at least one.  \vspace{3mm}

         In this case, \(\mathcal N_q^1F_{s_i}\) contains a unit vector \(v_*\) that is not
         parallel to \(v\). By Remark \ref{Rmk: normal cone convex} and Lemma
         \ref{Lem: nearest point preserving along normal direction}, for
         small \(\theta\) the curve
         \[
             \zeta(\theta)
             =
             \exp_q\left(
             \tau\frac{v+\theta v_*}{|v+\theta v_*|}
             \right)
         \]
         is contained in \(\Sigma_{i,\tau}\). Since
         \(\Sigma_{i,\tau}\) touches \(\partial D\) at \(p\) from
         outside, it follows from the Hessian comparison theorem and also the
         comparison principle that the largest principal curvature of
         \(\partial D\) at \(p\) with respect to the outward-pointing unit
         normal is no less than \(1/(2\tau)\) if \(\tau\) is small enough.

         On the other hand, by Lemma \ref{Lem: C11 hypersurface} and the
         comparison principle, we know that the smallest principal curvature
         of \(\partial D\) at \(p\) with respect to the outward-pointing unit
         normal is no less than \(-C\) for some universal constant \(C>0\)
         independent of \(\tau\in (0,\tau_*/2)\). From these facts, by taking
         \(\tau\) small enough we can guarantee
         \[
             H_{\partial D}(p)\geq \delta,
         \]
         which contradicts to our assumption \(H_{\partial D}(p)<\delta'\).
\end{proof}

Now, we are ready to prove Proposition \ref{Prop: mean-convex-smoothing}, where we will use the theory of soap bubbles to realize the second regularity improvement.

\begin{proof}[Proof of Proposition \ref{Prop: mean-convex-smoothing}]
By Lemma \ref{Lem: barrier Omega i,tau}, for each $l$ we can take some 
$$\Omega_{i,\tau}=\{q\in \Omega:\dist(q,F_{s_i})< \tau\}$$
 such that $\Sigma_{i,\tau}$ is a strict $(\delta-l^{-1})$-barrier in $\overline{\Omega}_{i,\tau}$, by taking $i$ large enough and $\tau$ small enough. Moreover, we can guarantee that
 $\Omega\setminus F_{s_i}$ lies in the $\min\{l^{-1},s_*\}$-neighborhood of $\Sigma$, where $s_*$ is the constant coming from \eqref{Eq: boundary distance gap}.
 
 Let $\sigma\in (0,\tau/2)$ be a constant. In the following, we will consider the band region
 $$V_\sigma=\overline\Omega_{i,\tau}\setminus\Omega_{i,\tau-\sigma}\mbox{ where }\partial V_\sigma=\Sigma_{i,\tau}\cup \Sigma_{i,\tau-\sigma}.$$
 It follows from the previous paragraph and Lemma \ref{Lem: C11 hypersurface} that 
 \begin{itemize}
 \item $\Sigma_{i,\tau}$ is a strict $(\delta-l^{-1})$-barrier in $V_\sigma$;
 \item $\Sigma_{i,\tau-\sigma}$ is a $(-C)$-barrier in $V_\sigma$ for some universal positive constant $C$ independent of $\sigma$.
 \end{itemize}
 Take a smooth function
 $$h_\sigma:V_\sigma\to [\delta-l^{-1},C+\delta]$$
 such that $h_\sigma\equiv \delta-l^{-1}$ around $\Sigma_{i,\tau}$ and $h_\sigma\equiv C+\delta$ around $\Sigma_{i,\tau-\sigma}$. Define
 $$\mathcal C_\sigma:=\{\mbox{Caccioppoli sets }U\mbox{ with }U\Delta\Omega_{i,\tau-\sigma/2}\subset V_\sigma\}$$
 and
 $$\mathcal A_\sigma(U)=\mathcal H^{n-1}(\partial^*U\cap V_\sigma)-\int_{U\cap V_\sigma}h_\sigma\,\mathrm d\mathcal H^n\mbox{ for }U\in \mathcal C_\sigma.$$
 
 \vspace{3mm}
 {\bf Claim 1}. There is a minimizer $U_\sigma$ of $\mathcal A_\sigma$ in $\mathcal C_\sigma$ such that $\partial U_\sigma\cap V_\sigma$ has mean curvature 
 $$H=h_\sigma|_{\partial U_\sigma}.$$
 
  \vspace{3mm}
 Let us follow the proof of \cite[Lemma 7]{Ilm}. For $\varepsilon$ small, we denote 
 $$V_{\sigma,\varepsilon}=\overline\Omega_{i,\tau-\varepsilon}\setminus\Omega_{i,\tau-\sigma+\varepsilon}.$$ The basic idea is to minimize the functional $\mathcal A_\sigma$ in the approximation class
 $$\mathcal C_{\sigma,\varepsilon}:=\{\mbox{Caccioppoli sets }U\mbox{ with }U\Delta\Omega_{i,\tau-\sigma/2}\subset V_{\sigma,\varepsilon}\}$$
 and to show that the minimizers of $\mathcal A_\sigma$ in the approximation class always avoid a fixed neighborhood of $\partial V_\sigma$.
 
From the same proof as Lemma \ref{Lem: barrier F_s}, we conclude that there is a constant $\varepsilon_0>0$ such that $\Sigma_{i,\tau-\varepsilon}$ and $\Sigma_{i,\tau-\sigma+\varepsilon}$ are strict $(\delta-l^{-1})$-barrier and strict $[-(C+\delta)]$-barrier in $V_{\sigma,\varepsilon}$, respectively, for all $\varepsilon\in (0,\varepsilon_0)$. From the geometric measure theory, there is a minimizer $U_{\sigma,\varepsilon}$ of $\mathcal A_\sigma$ in $\mathcal C_{\sigma,\varepsilon}$. 

In the following, we show that $\partial U_{\sigma,\varepsilon}$ always avoids the $\varepsilon_0$-neighborhood of $\partial V_\sigma$. Otherwise, there is a point 
$x\in \partial U_{\sigma,\varepsilon}$ such that
 $$\dist(x,\partial V_\sigma)=\dist(\partial U_{\sigma,\varepsilon},\partial V_\sigma)<\varepsilon_0.$$ 
Take a constant $r>0$ small enough such that the mean curvature of $\partial B_r(x)$ satisfies
\begin{equation}\label{Eq: mean curvature small ball}
H_{\partial B_r(x)}>\|h_\sigma\|_{L^\infty}.
\end{equation}
Then we minimize the functional $\mathcal A_\sigma$ among the Caccioppoli sets $U$ satisfying $U\Delta U_{\sigma,\varepsilon}\Subset B_r(x)$. Denote the minimizer by $\tilde U_{\sigma,\varepsilon}$. From \eqref{Eq: mean curvature small ball}, $\partial B_r(x)\setminus \partial U_{\sigma,\varepsilon}$ serves as barriers, and so we have $$\partial \tilde U_{\sigma,\varepsilon}\cap \partial B_r(x)=\partial U_{\sigma,\varepsilon}\cap \partial B_r(x).$$

We claim 
\begin{equation}\label{Eq: distance comparison}
\dist(\partial \tilde U_{\sigma,\varepsilon},\partial V_\sigma)\geq \dist(\partial U_{\sigma,\varepsilon},\partial V_\sigma).
\end{equation}
Otherwise, we can find a point $\tilde x\in \partial \tilde U_{\sigma,\varepsilon}\cap B_r(x)$ such that 
$$\tilde d:=\dist(\tilde x,\partial V_\sigma)=\dist(\partial \tilde U_{\sigma,\varepsilon},\partial V_\sigma)<\varepsilon_0.$$
In particular, we can find a small exterior or interior geodesic ball of $\tilde U_{\sigma,\varepsilon}$ touching $\partial \tilde U_{\sigma,\varepsilon}$ at the point $\tilde x$. Therefore, $\partial \tilde U_{\sigma,\varepsilon}$ is smooth around $\tilde x$ and from the first variation we have 
$$H_{\partial \tilde U_{\sigma,\varepsilon}}(\tilde x)=h_\sigma(\tilde x).$$
Since we have $$\delta-l^{-1}\leq h_\sigma\leq C+\delta,$$ this contradicts to the fact that $\Sigma_{i,\tau-\tilde d}$ and $\Sigma_{i,\tau-\sigma+\tilde d}$ are strict $(\delta-l^{-1})$-barrier and strict $[-(C+\delta)]$-barrier in $V_{\sigma,\tilde d}$, respectively.

From the distance comparison \eqref{Eq: distance comparison} we obtain $\tilde U_{\sigma,\varepsilon}\in \mathcal C_{\sigma,\varepsilon}$, and then from a direct comparison argument we see
$$\mathcal A_\sigma^{B_r(x)}(\tilde U_{\sigma,\varepsilon})\geq \mathcal A_\sigma^{B_r(x)}( U_{\sigma,\varepsilon}) \geq \mathcal A_\sigma^{B_r(x)}(\tilde U_{\sigma,\varepsilon}),$$
where 
$$\mathcal A_\sigma^{B_r(x)}( U)=\mathcal H^{n-1}(\partial^*U\cap B_r(x))-\int_{U\cap B_r(x)}h_\sigma\,\mathrm d\mathcal H^n.$$
This means that $U_{\sigma,\varepsilon}$ is a minimizer of $\mathcal A_\sigma$ without constraint in $B_r(x)$ as well. The same argument as we made for $\tilde U_{\sigma,\varepsilon}$ leads to a contradiction. 

We have proven that $\partial U_{\sigma,\varepsilon}$ always avoids the $\varepsilon_0$-neighborhood of $\partial V_\sigma$. Let us take $U_\sigma$ to be the limit of $U_{\sigma,\varepsilon}$ as $\varepsilon\to 0$ up to a subsequence. Then $U_\sigma$ is a minimizer of $\mathcal A_\sigma$ in $\mathcal C_{\sigma}$, which is also in $\mathcal C_{\sigma,\varepsilon_0}$. From the first variation, we conclude that $\partial U_\sigma\cap V_\sigma$ has mean curvature 
$$H=h_\sigma|_{\partial U_\sigma}.$$
 
 \vspace{3mm}
 {\bf Claim 2}. For $\sigma$ small enough, $\mathcal S_\sigma:=\partial U_\sigma\cap V_\sigma$ is smooth.

\vspace{3mm}
Recall that the mean curvature of $\mathcal S_\sigma$ is uniformly bounded by a constant independent of $\sigma$. By the Allard regularity from the geometric measure theory and the elliptic regularity theory, there are constants $r_0>0$ and $\theta_0>0$ (independent of $\sigma$) such that, for any $x\in \mathcal S_\sigma$, if we have 
\begin{equation}\label{Eq: density estimate}
\Theta(x,r):=\frac{\mathcal H^{n-1}(\mathcal S_\sigma\cap B_r(x))}{\omega_{n-1} r^{n-1}}<1+\theta_0
\end{equation} 
for some $r\in (0,r_0)$, then $\mathcal S_\sigma$ is smooth around $x$. Therefore, it suffices to establish the density estimate \eqref{Eq: density estimate} for any $x\in \mathcal S_\sigma$.

In the following, we are going to work in the Fermi coordinate. By Lemma \ref{Lem: nearest point preserving along normal direction} we have the one-to-one map
$$\Phi:\Sigma_{i,\tau}\times [\tau-\sigma,\tau]\to V_\sigma,\,(p,t)\mapsto \exp_{q_p}(tv_p),$$
where $(q_p,v_p)$ is the unique pair from Lemma \ref{Lem: unique pair} with $E=F_{s_i}$. Note that the map $\Phi$ is $C^1$ since it can be considered as the $C^1$-flow generated by the continuous vector field $V$ given by \eqref{Eq: continuous vector field}. Since the hypersurfaces $\Sigma_{i,t}$ with $t \in [\tau-\sigma,\tau]$ have uniform $C^{1,1}$-estimates, the metric $g$ has the form
$$g=\mathrm dt^2+g_t,$$
where $\{g_t\}_{t\in [\tau-\sigma,\tau]}$ is a Lipschitz family of metrics on $\Sigma_{i,\tau}$  satisfying
$$e^{-\Lambda(\tau-t)}g_\tau\leq g_t\leq e^{\Lambda(\tau-t)}g_\tau$$
for some universal constant $\Lambda>0$. We will also use the canonical projection map
$$\pi:V_\sigma\to \Sigma_{i,\tau}.$$
Since $\Sigma_{i,\tau}$ is a $C^{1,1}$-hypersurface, we are able to take a constant $r_1\in (0,r_0)$ such that the $g_\tau$-geodesic ball $D_{r_1}(y)$ on $\Sigma_{i,\tau}$ satisfies
$$\frac{\mathcal H^{n-1}_{g_\tau}(D_{r_1}(y))}{\omega_{n-1}r_1^{n-1}}\leq 1+\frac{\theta_0}{2}\mbox{ for any }y\in \Sigma_{i,\tau}.$$
Note that for any $x\in \partial U_\sigma\cap V_\sigma$ we have 
$$\pi(B_r(x)\cap V_\sigma)\subset D_{\alpha r}(y),$$
where $y=\pi(x)$ and $\alpha=e^{\Lambda\sigma}$. Take $r=r_1/\alpha$. By a comparison argument, we can derive
\[
\begin{split}
\mathcal H^{n-1}(\mathcal S_\sigma\cap B_r(x))
\leq &\,\mathcal H^{n-1}_{g_\tau}(D_{r_1}(y))+\mathcal H^{n-2}_{g_\tau}(\partial D_{r_1}(y))\cdot\sigma e^{(n-2)\Lambda \sigma/2}\\&+\|h_\sigma\|_{L^\infty}\cdot \mathcal H^{n-1}_{g_\tau}(D_{r_1}(y))\cdot\sigma e^{(n-1)\Lambda \sigma/2}.
\end{split}
\]
This yields the estimate (uniform for $x$)
$$\Theta(x,r)\leq 1+\theta_0/2+o(1)\mbox{ as }\sigma\to 0.$$
By taking $\sigma$ small enough, we obtain the desired density estimate \eqref{Eq: density estimate}, and so $\mathcal S_\sigma$ is smooth. This completes the proof for Claim 2.

\vspace{3mm}
We are ready to complete the proof by taking $$\Omega_l:=U_\sigma.$$ 
Let us verify the desired properties one by one. From our construction, we have $F_{s_i}\subset \Omega_l\subset F_{s_i'}$ for positive constants $s_i $ and $ s_i'$ less than $s_*:=\dist(\Sigma,\partial\Omega\setminus \Sigma)$. This implies $\partial\Omega_l\cap\partial\Omega=\partial\Omega\setminus \Sigma$ and $\Sigma_l:=\partial\Omega_l\cap\Omega=\mathcal S_\sigma$, which is a smooth hypersurface having mean curvature
$$H\geq \delta-l^{-1}$$
with respect to the outward-pointing unit normal. The same reason gives that $\Omega\setminus \Omega_l$ lies in the $l^{-1}$-neighborhood of $\Sigma$ by taking $s_i$ small enough. Let $\Sigma_\infty$ denote the Hausdorff limit of $\Sigma_l$ up to a subsequence. To conclude that $\Sigma_l$ converges to $\Sigma$ in the Hausdorff sense, it suffices to show $\Sigma_\infty=\Sigma$. On the one hand, since $\Sigma_l$ lies in the $l^{-1}$-neighborhood of $\Sigma$, we have $\Sigma_\infty\subset \Sigma$. On the other hand, since $\Sigma_l$ is homologous to $\Sigma$, the projection from $\Sigma_l$ to $\mathbf \Lambda_{\rho}(\Omega,\Sigma)$ is surjective when $l$ is large enough for any fixed $\rho>0$, and then we have $\mathbf\Lambda(\Omega,\Sigma)\subset \Sigma_\infty$. Since $\Sigma$ consists of boundary components of $\Omega$, we have $\overline{\mathbf\Lambda(\Omega,\Sigma)}=\Sigma$, and the previous discussion yields $\Sigma_\infty=\Sigma$ as desired.

It remains to show
$$\mathcal H^{n-1}(\Sigma_l)\to \mathcal H^{n-1}(\mathbf \Lambda(\Omega,\Sigma))\mbox{ as }l\to \infty.$$
By a diagonal argument, it suffices to check step by step the following facts:
\begin{itemize}
\item[(i)] $\mathcal H^{n-1}(\Sigma_{s_i})\to \mathcal H^{n-1}(\mathbf \Lambda(\Omega,\Sigma))$ as $i\to\infty$;
\item[(ii)] $\mathcal H^{n-1}(\Sigma_{i,\tau})\to \mathcal H^{n-1}(\Sigma_{s_i})$ as $\tau\to 0$;
\item[(iii)] and
$\mathcal H^{n-1}(\mathcal S_\sigma)\to\mathcal H^{n-1}(\Sigma_{i,\tau})$ as \(\sigma\to0\).
\end{itemize}

Let us prove (i). Since $\Sigma_{s_i}$ is a subset of the image under the exponential map from a smooth section of the normal bundle of $\mathbf\Lambda(\Omega,\Sigma)$, $\Sigma_{s_i}$ is $(n-1)$-rectifiable. On one hand, portions of $\Sigma_{s_i}$ can be written as smooth graphs over $\mathbf \Lambda_{\rho}(\Omega,\Sigma)$ for $i$ large enough, and so we have 
$$\liminf_{i\to \infty} \mathcal H^{n-1}(\Sigma_{s_i})\geq \mathcal H^{n-1}(\mathbf \Lambda_{\rho}(\Omega,\Sigma))\to \mathcal H^{n-1}(\mathbf \Lambda(\Omega,\Sigma)).$$
On the other hand, we have from the Riccati equation, the mean curvature of $\mathbf \Lambda_{\rho}(\Omega,\Sigma)$, and the area formula that
$$\mathcal H^{n-1}(\Sigma_{s_i})\leq e^{C_*s_i^2}\cdot\mathcal H^{n-1}(\mathbf \Lambda_{s_i}(\Omega,\Sigma)),$$
where $C_*$ denotes the bound of the Ricci curvature around $\Sigma$. Clearly, we have 
$$\limsup_{i\to \infty} \mathcal H^{n-1}(\Sigma_{s_i})\leq \mathcal H^{n-1}(\mathbf \Lambda(\Omega,\Sigma)),$$
and then (i) follows. 

For (ii), we use the following Steiner formula \cite[Theorem 5.6]{Fed59} for sets with positive reach established by Federer
$$\mathcal H^n(\Omega_{i,\tau}\setminus F_{s_i})=\sum_{j=0}^n\tau^{n-j}|B_1^{n-j}|\psi_j(F_{s_i}),$$
where $\psi_j$ are curvature measures.
By taking derivative, we have
$$\mathcal H^{n-1}(\Sigma_{i,\tau})=\sum_{j=0}^{n-1}\tau^{n-j-1}(n-j)|B_1^{n-j}|\psi_j(F_{s_i}),$$
and so
$$\mathcal H^{n-1}(\Sigma_{i,\tau})\to 2\psi_{n-1}(F_{s_i})\mbox{ as }\tau \to 0.$$
 By \cite[Theorem 8]{ACV08} we know
\[
  2\psi_{n-1}(F_{s_i})=  \int_{\Sigma_{s_i}}
    \mathcal H^0\bigl(\mathcal N_q^1F_{s_i}\bigr)
    \,d\mathcal H^{n-1}(q).
\]
From \cite[Remark 4.15]{Fed59}, we know that $\mathcal N_qF_{s_i}$ has dimension one $\mathcal H^{n-1}$-almost everywhere. Combined with the proof of Lemma
\ref{Lem: barrier Omega i,tau}, 
\(
    \mathcal N_q^1F_{s_i}
\) contains only one unit vector. Therefore, we obtain
\[
    \mathcal H^{n-1}(\Sigma_{i,\tau})
    \to
    2\psi_{n-1}(F_{s_i})=\mathcal H^{n-1}(\Sigma_{s_i})
\]
as desired.

It remains to prove (iii). Comparing the minimizer \(U_\sigma\) with \(\Omega_{i,\tau}\) gives
\[
    \mathcal H^{n-1}(\mathcal S_\sigma)
    \leq
    \mathcal H^{n-1}(\Sigma_{i,\tau})
    +
    \|h_\sigma\|_{L^\infty}\cdot\mathcal H^n(V_\sigma).
\]
Since \(\mathcal H^n(V_\sigma)\to0\) as $\sigma\to 0$, this gives 
\[
\limsup_{\sigma\to 0}\mathcal H^{n-1}(\mathcal S_\sigma)\leq \mathcal H^{n-1}(\Sigma_{i,\tau}).
\]
On the other hand, we can estimate the lower bound for \(\mathcal H^{n-1}(\mathcal S_\sigma)\). Use the projection
\[
    \pi:V_\sigma\to\Sigma_{i,\tau}
\]
from Claim 2. Since \(\mathcal S_\sigma\) separates the two boundary components of \(V_\sigma\), its image under the projection map is the entire \(\Sigma_{i,\tau}\). Using the fact
\[
    e^{-\Lambda(\tau-t)}g_\tau\leq g_t\leq e^{\Lambda(\tau-t)}g_\tau \mbox{ for } t\in[\tau-\sigma,\tau],
\]
we obtain
\[
    \mathcal H^{n-1}(\mathcal S_\sigma)
    \geq
    e^{-\frac{(n-1)\Lambda\sigma}{2}}
    \mathcal H^{n-1}(\Sigma_{i,\tau}).
\]
This yields
\[
\liminf_{\sigma\to 0}\mathcal H^{n-1}(\mathcal S_\sigma)\geq \mathcal H^{n-1}(\Sigma_{i,\tau}),
\]
and we complete the proof of (iii).
 \end{proof}

The preceding smoothing method also gives another solution to Lawson's mean-convex approximation problem. Lawson asked \cite[Problem~5.7]{Bro86} whether, for every stable minimal hypercone
$C\subset\mathbb{R}^n$ and every $\varepsilon>0$, there exists a properly embedded smooth hypersurface with positive mean curvature whose intersection with $B_1(0)$ is within Hausdorff distance $\varepsilon$ of $C\cap B_1(0)$. Wang \cite[Theorems~1.1--1.3]{Wan24} answered this affirmatively using global one-sided constructions based on minimizing hypersurfaces and mean-convex self-expanders. The following corollary obtains the same local approximation for stationary boundaries and requires neither stability nor a conical structure.

\begin{corollary}
Let $B'\Subset B$ be Euclidean balls and let $E\subset\mathbb{R}^n$ be an open set of locally finite perimeter. Suppose that $\partial E$ is stationary in $B$. Then there exist open sets $E_j\subset E$ such that $\partial E_j\cap B'$ is smooth and
strictly mean-convex with respect to the outward-pointing unit normal. Moreover, $\chi_{E_j}\to\chi_E$ in $L^1(B')$, $\partial E_j\to\partial E$ in the Hausdorff sense on $\overline{B'}$, and $\mathcal{H}^{n-1}(\partial E_j\cap B')\to \mathcal{H}^{n-1}(\partial E\cap B')$.
\end{corollary}
\begin{proof}
Choose a ball $B_0$ with $B'\Subset B_0\Subset B$ such that $\partial B_0$ is transverse to $\partial^* E$, and set $\Omega=E\cap B_0$. By \cite[Lemma~1]{Ilm}, $\partial E\cap B$ is a $0$-barrier, whereas $\partial B_0$ is strictly mean-convex. Hence $\partial\Omega$ is a $0$-barrier in $\overline{\Omega}$.

The same interior tangent-ball and blow-up argument as in the proof of Proposition~2.1, using the stationary half-space property \cite[Lemma~10]{Ilm}, shows that every point of $\mathbf\Lambda(\Omega,\partial\Omega)\cap\partial E$ is regular. The remaining visible points lie on the smooth part of $\partial B_0$, while the transverse corner is not visible from $\Omega$. Thus $\partial\Omega$ is compact and quasi-regular relative to $\Omega$.

Take $u(x)=1+|x|^2$ on a neighborhood of $\overline{\Omega}$, and let $d_u$ be the distance induced by $u^{-2}g_{\mathrm{Euc}}$. Since
\[
Ku=\inf_{|v|=1}\bigl(\Delta u-D^2u(v,v)\bigr)=2(n-1)>0,
\]
\cite[Lemma~6, in particular (23)]{Ilm} shows that
\[
\Omega_r=\{x\in\Omega:d_u(x,\partial\Omega)>r\}
\]
has strict $0$-barrier boundary in $\overline{\Omega_r}$ for every sufficiently small $r>0$.

The argument of Lemma~3.19 also applies to $d_u$, which is  semiconcave. We may therefore choose $r_j\downarrow0$ such that $\overline{\Omega_{r_j}}$ has positive reach. Applying the smoothing argument in the proof of Proposition~2.6 to $\Omega_{r_j}$ and choosing the approximation parameters diagonally, we obtain open sets $E_j\subset\Omega_{r_j}\subset E$ whose boundaries are smooth and strictly mean-convex in $B'$ and for which all the stated convergences hold.
\end{proof}

\begin{remark}
In Euclidean space, the restriction to balls is used only to produce a compact auxiliary boundary to which the preceding argument applies directly. Localizing the positive-reach regularization and subsequent smoothing gives the same conclusion on arbitrary relatively compact open subsets. The same argument works on sufficiently small Riemannian balls. Indeed, since $D^2d_p^2=2g+O(\rho^2)$ and $\operatorname{Ric}$ is locally bounded, $u=1+A d_p^2$ satisfies $Ku>0$ on $B_\rho(p)$ for $A$ sufficiently large and then $\rho$ sufficiently small.
\end{remark}

\section{Escaping geodesic line}

The goal of this section is to prove Theorem \ref{Thm: escaping geodesic line}. Let us briefly recall the set-up.
Let $(M,g)$ be a complete non-compact Riemannian manifold with a fixed increasing compact exhaustion $\{K_l\}_{l\geq 1}$. A collection 
$$E=\{V_l\}_{l\geq 1}$$
is called an end of $M$ if each $V_l$ is an unbounded component of the complement $M\setminus K_l$ and we have $V_{l+1}\subset V_l$ for all $l\geq 1$. Given any geodesic line 
$$\gamma:(-\infty,+\infty)\to (M,g),$$ we will abuse the notation
$$\gamma(+\infty)\in E$$
to mean that there is a sequence of positive constants $s_l\uparrow +\infty$ such that 
$$\gamma([s_l,+\infty))\subset V_l$$
for all $l\geq 1$. We also abuse the notation $\gamma(-\infty)\in E$ in a similar way.

First we recall the following well-known lemma.
\begin{lemma}\label{Lem: two ends geodesic}
If $(M,g)$ has two different ends, denoted by $E_-$ and $E_+$, then there is a geodesic line
$$\gamma:(-\infty,+\infty)\to (M,g)$$
such that $\gamma(-\infty)\in E_-$ and $\gamma(+\infty)\in E_+$.
\end{lemma}
\begin{proof}
Denote
$$E_-=\{V_{l}^-\}_{l=1}^\infty\mbox{ and }E_+=\{V_{l}^+\}_{l=1}^\infty.$$
Then we can take two sequences of points
$$\{p_l^-\}_{l=1}^\infty\mbox{ and } \{p_l^+\}_{l=1}^\infty$$
 such that $p_l^-\in V_{l}^-$ and $p_l^+\in V_{l}^+$. It follows from the Hopf-Rinow theorem that there are unit-speed minimizing geodesic segments
$$\gamma_l:[-s_l',s_l'']\to (M,g)$$
connecting $p_l^-$ to $p_l^+$. Let $l_0$ denote the smallest index among those $l$, where $V_{l,-}$ and $V_{l,+}$ are different. Then $\gamma_l$ must intersect with $K_{l_0}$ for all $l\geq l_0$, and so we can always make the normalization $\gamma_l(0)\in K_{l_0}$. Since the endpoints $p_l^-$ and $p_l^+$ diverge to the infinity, we have $s_l',s_l''\to+\infty$. Up to a subsequence, $\gamma_l$ converge to a geodesic line
$\gamma:(-\infty,+\infty)\to (M,g)$ in $C^\infty_{loc}$-sense with $\gamma(0)\in K_{l_0}$.

Next let us show $\gamma(-\infty)\in E_-$. Fix $i\geq l_0$ and we are going to determine $s_i$ such that 
$\gamma((-\infty,-s_i])\subset V_{i}^-$. Take $s_i$ to be the diameter of $K_i$. Since we have $\gamma_l(0)\in K_{l_0}\subset K_i$ for all $l\geq i$, the geodesic segment $\gamma_l([-s_l',-s_{i}])$ lies in the same component of $M\setminus K_{i}$. Note that we have 
$$\gamma_l(-s_l')\in V_{l}^-\subset V_i^-.$$
Therefore, we know 
$$\gamma_l([-s_l',-s_{i}])\subset V_{i}^-.$$
As the limit of $\gamma_l(-s_i-1)$, we have $\gamma(-s_i-1)\in V_i^-$, and the same argument yields $\gamma((-\infty,-s_i])\subset V_{i}^-$ as desired.

Similarly, we have $\gamma(+\infty)\in E_+$, and the proof is completed.
\end{proof}

We are now ready to prove Theorem \ref{Thm: escaping geodesic line}.
\begin{proof}[Proof of Theorem \ref{Thm: escaping geodesic line}]
Since $e(M)=+\infty$, we can take a sequence of ends
$$\{E_i\}_{i=1}^\infty,$$
which are pairwise different. For each $l\geq 1$, by Lemma \ref{Lem: two ends geodesic} we can construct geodesic lines 
$$\gamma_l:(-\infty,+\infty)\to (M,g)$$
such that $\gamma_{l}(-\infty)\in E_{2l-1}$ and $\gamma_l(+\infty)\in E_{2l}$. If there is some geodesic line $\gamma_l$ not touching $K$, then we are done. Then we just need to consider the case where all the geodesic lines $\gamma_l$ intersect with $K$.

For normalization, we may assume $\gamma_l(0)\in K$ for all $l$. Up to subsequence, $\gamma_l$ converge to a limit geodesic line
$$\gamma_\infty:(-\infty,+\infty)\to (M,g)\mbox{ with }\gamma_\infty(0)\in K.$$
Let $D$ be any positive constant. Then we can take indices $l_1$ and $l_2$ large enough such that
$$\dist(\gamma_{l_1}(D), \gamma_{l_2}(D))<1.$$
Take a sequence of positive constants $s_i\to +\infty$ as $i\to \infty$. Let $\zeta_i$ be a minimizing geodesic segment connecting $\gamma_{l_1}(s_i)$ and $\gamma_{l_2}(s_i)$.

\vspace{3mm}
 {\bf Claim}. If $D$ is large enough, then $\zeta_i$ does not touch the $1$-neighborhood of $K$.
 
  \vspace{3mm}
  
  Suppose not, then $\zeta_i$ has the length estimate
  $$L(\zeta_i)\geq  \dist(\gamma_{l_1}(s_i),K)+\dist(\gamma_{l_2}(s_i),K)-2\geq2(s_i-\diam K-1).$$
  On the other hand, we have the estimate
  $$L(\zeta_i)\leq 2(s_i-D)+1.$$
  By taking the constant $D\geq \diam K+2$, we obtain a contradiction.
  
  \vspace{3mm}
  As in the proof of Lemma \ref{Lem: two ends geodesic}, the geodesic segments $\zeta_i$ converge to a geodesic line $\zeta$ with $\zeta(-\infty)\in E_{2l_1}$ and $\zeta(+\infty)\in E_{2l_2}$ up to a subsequence, which lies outside $1$-neighborhood of $K$ from the claim above. The proof is completed.
\end{proof}

\section{Applications}
\subsection{Revisit results for Ricci curvature}
In this subsection, we will give new proofs for two classical theorems with nonnegative Ricci curvature based on our curvature-free results. 

First, let us focus on the minimal volume growth theorem for complete non-compact Riemannian manifolds with nonnegative Ricci curvature.
\begin{theorem}[Calabi, Yau]\label{Thm: Calabi-Yau}
Let $(M,g)$ be a complete and non-compact Riemannian manifold with nonnegative Ricci curvature, then $(M,g)$ has at least linear volume growth. That is, we have
$$\liminf_{r\to +\infty}\frac{\vol_g(B_r(p))}{r}>0$$
for all $p\in M$.
\end{theorem}

Our new proof is based on the existence of mean-concave region.
\begin{proof}
Suppose by contradiction that we have
$$\liminf_{r\to +\infty}\frac{\vol_g(B_r(p))}{r}=0$$
for some $p\in M$. It follows from Theorem \ref{Thm: mean concave exhaustion} that there is a smooth bounded region $U$ such that $\partial U$ is mean-concave with respect to the unit normal pointing to the infinity. Take
$$E=M\setminus U.$$
Then $E$ is a complete and non-compact Riemannian manifold with compact mean-convex boundary, which has nonnegative Ricci curvature. It follows from the half-cylinder splitting theorem proved by Croke and Kleiner \cite{CK92} that $E$ splits isometrically as $$(N,h)\times [0,+\infty)$$ for some closed Riemannian manifold $(N,h)$. In particular, we have
$$\liminf_{r\to +\infty}\frac{\vol_g(B_r(p))}{r}>0,$$
which gives the desired contradiction.
\end{proof}

Next, let us focus on the finite-ends theorem.

\begin{theorem}[Cai, Li--Tam]
Let $(M,g)$ be a complete Riemannian manifold with nonnegative Ricci curvature outside a compact subset. Then $(M,g)$ has finitely many ends.
\end{theorem}

Our new proof is based on the existence of escaping geodesic lines.

\begin{proof}
Suppose, to the contrary, that \(e(M)=+\infty\). By \cite{Ehrlich76I}, we can deform the metric $g$ slightly such that $(M,g)$ has positive Ricci curvature outside a compact subset $K$. By Theorem \ref{Thm: escaping geodesic line} there is a geodesic line $\gamma$ outside $K$. But from the second variation we can deduce that the tangential Ricci curvatures $\Ric(\gamma',\gamma')$ vanish along $\gamma$, which leads to a contradiction.
\end{proof}

\subsection{New results for scalar and holomorphic sectional curvatures}
In this subsection, we will extend the idea above to establish a minimal volume growth theorem for scalar curvature, and also a finite-ends theorem for holomorphic sectional curvature.

\subsubsection{Minimal volume growth theorem for scalar curvature}

We will derive the volume-growth statement from the following half-cylinder splitting theorem.

\begin{proposition}\label{Prop: half-splitting}
Let $(M^n,g)$, $2\leq n\leq 7$, be a complete non-compact Riemannian manifold with compact mean-convex boundary. If $(M,g)$ has nonnegative scalar curvature and any hypersurface separating the boundary $\partial M$ to the infinity admits no smooth metric with positive scalar curvature, then $(M,g)$ splits isometrically as the Riemannian product $$(N,h)\times [0,+\infty),$$
where $(N,h)$ is a closed Riemannian manifold.
\end{proposition}

\begin{proof}
If $g$ is Ricci-flat or \(n=2\), then the consequence follows directly from the half-cylinder splitting theorem by Croke and Kleiner  \cite{CK92}. Otherwise, we can deform $g$ slightly as in \cite{Kazdan82} such that $g$ has positive scalar curvature and the boundary is still mean-convex. 

In the following, we are going to construct a separating hypersurface  based on soap bubbles, which admits a smooth metric positive scalar curvature.
By modifying the distance function to $\partial M$, we can construct a proper smooth function $$\rho:M\to [0,+\infty)$$ such that $\rho^{-1}(0)=\partial M$, $\rho(x)\to +\infty$ as $x\to \infty$, and $\Lip\rho<1$. 
Similar as \cite[Lemma 2.3]{Zhu23}, we can construct a smooth function 
$$h:[0,L)\to (-\infty,0)$$ 
such that $h<0, h'<0$, $h(t)\to -\infty$ as $t\to L$, and 
\begin{equation}\label{Eq: modified scalar curvature}
\frac{n}{n-1}h^2+2h'+\min_{\rho^{-1}([0,1])}R_g\cdot \chi_{[0,1]}>0.
\end{equation}
Let $\Sigma$ be any hypersurface in $\rho^{-1}([0,L))$ separating $\partial M$ to the infinity, and $\Omega$ be the region enclosed by $\Sigma$ and $\partial M$. We consider the functional
$$\mathcal A(\Sigma,\Omega)=\mathcal H^{n-1}(\Sigma)-\int_\Omega h\circ \rho\,\mathrm d\mathcal H^n.$$
From the geometric measure theory, when $n\leq 7$ we can find a smooth minimizer $(\Sigma_*,\Omega_*)$ of the functional $\mathcal A$. In particular, (each component of) $\Sigma_*$ satisfies $H=h\circ\rho$ and the following stable inequality
$$\int_{\Sigma_*}|\nabla\phi|^2\,\mathrm d\sigma\geq \int_{\Sigma_*}(\Ric(\nu,\nu)+|A|^2+\partial_\nu(h\circ \rho))\phi^2\,\mathrm d\sigma.$$
Using the Schoen--Yau rearrangement, we can derive
\[
\begin{split}
\int_{\Sigma_*}|\nabla\phi|^2+\frac{1}{2}&R_{\Sigma_*}\phi^2\,\mathrm d\sigma\\
&\geq \frac{1}{2}\int_{\Sigma_*}\left(R_g+\left(\frac{n}{n-1}h^2+2h'\right)\circ\rho\right)\phi^2\,\mathrm d\sigma.
\end{split}
\]
Note that the right-hand side is positive due to \eqref{Eq: modified scalar curvature}.
Then we can construct a smooth metric on $\Sigma_*$ with positive scalar curvature as follows:
\begin{itemize}
\item If $\dim \Sigma_*=2$, then by taking $\phi\equiv 1$ we see that $\Sigma_*$ consists of spheres from the Gauss-Bonnet formula. Therefore, we can take the standard spherical metric as the desired smooth metric with positive scalar curvature.
\item If  $\dim \Sigma_*\geq 3$, then we have
$$\int_{\Sigma_*}|\nabla\phi|^2+\frac{\dim \Sigma_*-2}{4(\dim \Sigma_*-1)}R_{\Sigma_*}\phi^2\,\mathrm d\sigma>0.$$
Let $u$ be the first eigenfunction of the quadratic form above, and we take the conformal metric
$$u^{\frac{4}{\dim\Sigma_*-2}}g_{\Sigma_*},$$
which has positive scalar curvature.
\end{itemize}
This contradicts to our assumption.
\end{proof}

We are ready to prove Theorem \ref{Thm: minimal volume scalar}, Corollary \ref{Cor: contractible}, and Corollary \ref{Cor: aspherical at infinity}.
\begin{proof}[Proof of Theorem  \ref{Thm: minimal volume scalar}]
The proof is exactly the same as that of Theorem \ref{Thm: Calabi-Yau}. Suppose by contradiction that we have
$$\liminf_{r\to +\infty}\frac{\vol_g(B_r(p))}{r}=0$$
for some $p\in M$. It follows from Theorem \ref{Thm: mean concave exhaustion} that there is a smooth bounded region $U$ containing $K$ such that $\partial U$ is mean-concave with respect to the unit normal pointing to the infinity. Take
$$E=M\setminus U.$$
Then $E$ is a complete non-compact Riemannian manifold with compact mean-convex boundary, which has nonnegative scalar curvature. From the assumption, any hypersurface separating $\partial U$ (and hence $K$) to the infinity admits no smooth metric with positive scalar curvature. Then it follows from Proposition \ref{Prop: half-splitting} that $E$ splits isometrically as $$(N,h)\times [0,+\infty)$$ for some closed Riemannian manifold $(N,h)$. In particular, we have
$$\liminf_{r\to +\infty}\frac{\vol_g(B_r(p))}{r}>0,$$
which gives the desired contradiction.
\end{proof}

\begin{proof}[Proof of Corollary \ref{Cor: contractible}]
To show that $M$ is homeomorphic to $\mathbb R^3$, it suffices to prove that $M$ is simply-connected at infinity (see \cite{Stalling71}). That is, for any compact subset $K$ we can find a larger compact subset $\tilde K$ such that the inclusion map 
$$i_*:\pi_1(M\setminus\tilde K)\to \pi_1(M\setminus K)$$ is the zero map.

Since $(M,g)$ has sublinear lower volume growth, by Theorem \ref{Thm: minimal volume scalar} we can find smooth bounded exhaustion $$\{V_l\}_{l=1}^\infty$$ such that $\partial V_l$ consists of spheres. 
For any compact subset $K$, we can take some $V_l$ containing $K$ and choose $\tilde K$ to be a fixed compact subset containing $V_l$. Take any loop $\gamma$ in $M\setminus\tilde K$. Since $M$ is contractible, $\gamma$ bounds a disk $D$ in $M$, which may pass through $K$ a priori. We explain how to modify $D$ to be a new disk outside $K$. By slight perturbation, we may assume that $D$ is transversal to $\partial V_l$, and so $D$ intersects $\partial V_l$ along finitely many circles. These circles bound disks $D_1,\ldots,D_j$ in $D$ and bound $ D_1^*,\ldots, D_j^*$ in $\partial V_l$. We modify $D$ to be
$$\left(D\setminus\left(\bigcup D_i\right)\right)\cup \left(\bigcup D_i^*\right),$$
which is a disk enclosed by $\gamma$ outside $K$. Hence, $M$ is simply connected at infinity, and is homeomorphic to $\mathbb R^3$.
\end{proof}

\begin{proof}[Proof of Corollary \ref{Cor: aspherical at infinity}]
It follows from the previous work \cite{He-Zhu25} that any hypersurface separating $K$ from the infinity cannot admit any smooth metric with positive scalar curvature. The corollary follows from Theorem \ref{Thm: minimal volume scalar}.
\end{proof}

\subsubsection{Finite-ends theorem for holomorphic sectional curvature}

Let us recall some fundamental concepts in Kähler geometry that will be used throughout this section.

A K\"ahler manifold $(M, g, J)$ means a complex manifold $M$ equipped with a Riemannian metric $g$ and a complex structure $J$ that satisfy the following conditions:
\begin{itemize}
\item the complex structure is compatible with the metric, that is, we have  $g(JX, JY) = g(X, Y)$ for all tangent vector fields $X, Y$;
\item the complex structure is parallel, that is, we have $\nabla J = 0$, where $\nabla$ is the Levi-Civita connection of the metric $g$.
\end{itemize}

Given any point $p \in M$ and any unit vector $v \in T_pM$ on a Kähler manifold $(M, g, J)$, the holomorphic sectional curvature $H(v)$ is defined as
\[
H(v) = \Rm(v, Jv, v, Jv)
\]
where $\Rm$ denotes the Riemann curvature tensor of the metric $g$. We say that a K\"ahler manifold has positive holomorphic sectional curvature if we have $H(v)>0$ for all unit vector $v\in T_pM$ at any point $p\in M$.

Let us prove Theorem \ref{Thm: finite ends holomorphic}.
\begin{proof}[Proof of Theorem \ref{Thm: finite ends holomorphic}]
Suppose, to the contrary, that $e(M)=+\infty$. By Theorem \ref{Thm: escaping geodesic line} there is a geodesic line $\gamma$ outside $K$. Since $M$ is Kähler, $\nabla J=0$, and hence
$J\gamma'$ is a parallel variation field along $\gamma$. Therefore, by the
second variation formula, the holomorphic sectional curvature $H(\gamma')$
vanishes along $\gamma$, which leads to a contradiction.
\end{proof}

\appendix
\section{Finite-volume complete manifolds with nonnegative scalar curvature}\label{sec: finite volume}
In this section, we construct a conformally flat metric on $\mathbb R^n$, $n\geq 3$, with nonnegative scalar curvature but finite volume, which implies that additional topological assumptions are necessary for the validity of the minimal volume growth theorem. 
\begin{lemma}
	There is a complete metric $g$ on $\mathbb R^n$, $n\geq 3$, with nonnegative scalar curvature such that $(M,g)$ has finite volume.
\end{lemma}
\begin{proof}
	Let us consider a conformally flat metric $g=u^{\frac{4}{n-2}}g_{euc}$ with $u$ a smooth positive function on $\mathbb R^n$ to be determined. To ensure $(\mathbb R^n,g)$ having nonnegative scalar curvature we just need to guarantee $\Delta u\leq 0$ concerning the formula
	$$
	-\Delta u=c_nR(g)u^{\frac{n+2}{n-2}}.
	$$
	
	The desired function $u$ is constructed as follows. We start with a function
	$$v=\left(r\ln r\right)^{-\frac{n-2}{2}}.$$
	A straight-forward computation gives
	$$
	\Delta v=-\frac{n-2}{2}\left(r\ln r\right)^{-\frac{n+2}{2}}\left(\frac{n-2}{2}\ln^2 r-\frac{n}{2}\right).
	$$
	%Detailed computation for $\Delta v$:
%$$
%	v_i=-\frac{n-2}{2}\left(r\ln r\right)^{-\frac{n}{2}}(\ln r+1)\frac{x_i}{r}
%	$$
%	$$
%	-\frac{2}{n-2}v_{ii}=-\frac{n}{2}\left(r\ln r\right)^{-\frac{n+2}{2}}(\ln r+1)^2\frac{x_i^2}{r^2}+\left(r\ln r\right)^{-\frac{n}{2}}\frac{x_i^2}{r^3}+\left(r\ln r\right)^{-\frac{n}{2}}(\ln r+1)\frac{r^2-x_i^2}{r^3}
%	$$
%	$$
%	-\frac{2}{n-2}\left(r\ln r\right)^{\frac{n+2}{2}}\Delta v=-\frac{n}{2}(\ln r+1)^2+\ln r+(n-1)\ln r(\ln r+1)=\frac{n-2}{2}\ln^2 r-\frac{n}{2}
%	$$
In particular, there is an absolute constant $r_0$ such that $\Delta v<0$ when $r\geq r_0$. Denote $v_0=v(r_0)$. To do composition we have to construct a function $\zeta:[0,+\infty)\to [0,v_0/2]$ satisfying
\begin{itemize}
	\item $\zeta(t)\equiv t$ in a neighborhood of $0$ and $\zeta(t)\equiv const.$ when $t\geq v_0/2$;
	\item $\zeta'(t)\geq 0$ and $\zeta''(t)\leq 0$ for all $t\geq 0$.
\end{itemize}
Such function can be constructed from integration. Take a nonnegative monotone-decreasing function $\eta:[0,+\infty)\to [0,1]$ such that $\eta\equiv 1$ in $[0,v_0/4]$ and $\eta\equiv 0$ in $[v_0/2,+\infty)$. It suffices to define
$$
\zeta(t)=\int_0^t\eta(s)\,\mathrm ds.
$$
Let $u=\zeta\circ v$. Note that $u$ is defined on the whole $\mathbb R^n$ since it is constant in the $r_0$-ball. It is direct to compute
$$
\Delta u=\zeta''|\nabla v|^2+\zeta'\Delta v.
$$
When $r\geq r_0$ it follows from $\Delta v<0$ and the construction of $\zeta$ that $\Delta u\leq 0$. When $r\leq r_0$ we simply have $\Delta u\equiv 0$ due to its constancy.

It remains to verify the completeness and the finite volume of $(\mathbb R^n,g)$. To see the completeness we compute
\begin{equation*}
	\begin{split}
		\dist(O,\infty)&=\int_{0}^{+\infty}u^{\frac{2}{n-2}}\,\mathrm dr\\
		&\geq \int_{r_0}^{+\infty}\frac{1}{r\ln r}\,\mathrm dr=\left.\ln\ln r\right|_{r_0}^{+\infty}=+\infty.
	\end{split}
\end{equation*}
On the other hand, the volume can be computed as
\begin{equation*}
	\begin{split}
		\vol(\mathbb R^n,g)&=\int_{\mathbb R^n}u^{\frac{2n}{n-2}}\,\mathrm dx\\
		&\leq \omega_n r_0^n\left(\frac{v_0}{2}\right)^{\frac{2n}{n-2}}+n\omega_n\int_{r_0}^{+\infty}\frac{1}{r\ln^n r}\,\mathrm dr\\
		&<+\infty.
	\end{split}
\end{equation*}
This completes the proof.
\end{proof}

\bibliographystyle{alpha}

\bibliography{bib}
\end{document}